%% file: main.tex
\documentclass[12pt]{amsart}
\usepackage{amsthm}
\usepackage{amssymb}
\usepackage{amsmath}
\usepackage{parskip}
\usepackage{tikz}
\usepackage{tikz-cd}
\usepackage{bbm}
\usepackage{enumitem}

\usepackage{parskip}
\usepackage{graphicx}
\usepackage{import}
\usepackage{yhmath}
\usepackage{aliascnt}
\usepackage{standalone}
\usepackage{caption}
\usepackage{subcaption}

\usepackage[colorlinks=true, linkcolor=cyan, anchorcolor=cyan,citecolor=cyan,filecolor=cyan,menucolor=cyan,runcolor=cyan,urlcolor=cyan]{hyperref}

\usepackage{mathtools}

\usepackage[margin=1in,marginparwidth=0.8in, marginparsep=0.1in]{geometry}

\usepackage[draft]{say}
\definecolor{egyptianblue}{rgb}{0.06, 0.2, 0.65}

\usepackage[dvipsnames]{xcolor}
\definecolor{MHcol}{RGB}{16, 195, 235}

\definecolor{JBcol}{RGB}{109, 50, 168}

\usepackage{zref-clever}
\zcsetup{cap}

\newtheorem{theorem}{Theorem}[section]

\newtheorem{lemma}[theorem]{Lemma}
\AddToHook{env/lemma/begin}{%
   \zcsetup{countertype={theorem=lemma}}}
\zcRefTypeSetup{lemma}{Name-sg=Lemma}

\newtheorem{corollary}[theorem]{Corollary}
\AddToHook{env/corollary/begin}{%
   \zcsetup{countertype={theorem=corollary}}}
\zcRefTypeSetup{corollary}{Name-sg=Corollary}

\newtheorem{proposition}[theorem]{Proposition}
\AddToHook{env/proposition/begin}{%
   \zcsetup{countertype={theorem=proposition}}}
\zcRefTypeSetup{proposition}{Name-sg=Proposition}

\theoremstyle{definition}
\newtheorem{definition}[theorem]{Definition}
\AddToHook{env/definition/begin}{%
   \zcsetup{countertype={theorem=definition}}}
\zcRefTypeSetup{definition}{Name-sg=Definition}

\AddToHook{env/example/begin}{%
   \zcsetup{countertype={theorem=example}}}
\zcRefTypeSetup{example}{Name-sg=Example}

\newtheorem{remark}[theorem]{Remark}
\AddToHook{env/remark/begin}{%
   \zcsetup{countertype={theorem=remark}}}
\zcRefTypeSetup{remark}{Name-sg=Remark}

\newtheorem{construction}[theorem]{Construction}
\AddToHook{env/construction/begin}{%
   \zcsetup{countertype={theorem=construction}}}
\zcRefTypeSetup{construction}{Name-sg=Construction}

\AddToHook{env/claim/begin}{%
   \zcsetup{countertype={theorem=claim}}}
\zcRefTypeSetup{claim}{Name-sg=Claim}

\input{macros}

\DeclareRobustCommand{\SkipTocEntry}[5]{}

\newcommand{\red}[1]{{\color{red}#1}}

\newcommand{\norm}[1]{\left\Vert #1 \right\Vert}

\newcommand{\R}{\mathbb{R}}
\newcommand{\C}{\mathbb{C}}
\newcommand{\Z}{\mathbb{Z}}

\renewcommand{\P}{\mathbb{P}}

\newcommand{\Hom}{\mathrm{Hom}}

\newcommand{\dd}{\mathrm{d}}

\title{Mirror Symmetry for the Painlevé  Character Varieties}

\author{Joël Beimler, Mingyuan Hu, William Olsen, Vivek Shende}

\begin{document}
\begin{abstract}
We establish a homological mirror  theorem for the 4-manifolds arising as moduli of (irregular) rank two local systems on the projective line.  Specifically, we prove that the Fukaya category of  a moduli of such local systems with  generic microlocal monodromy at punctures is equivalent to the category of coherent sheaves on the minimal resolution of the corresponding moduli of local systems with trivial microlocal monodromy. 
\end{abstract}

\maketitle



\input{sections/introduction.tex}
\input{sections/A-side_handleslide}

 \clearpage
\input{sections/mirrorsymmetry}
\input{sections/B-side_toric_model}

\appendix

\input{sections/degenerations}

\newpage


\bibliographystyle{plain}
\bibliography{biblio}

\end{document}

%% file: macros.tex
\newcommand{\bA}{\mathbb{A}}
\newcommand{\bP}{\mathbb{P}}

\newcommand{\bC}{\mathbb{C}}

\newcommand{\bR}{\mathbb{R}}

\newcommand{\bM}{\mathbf{M}}
\newcommand{\bZ}{\mathbb{Z}}

\newcommand{\fp}{\mathfrak{p}}

\newcommand{\cO}{\mathcal{O}}

\newcommand{\PI}{{ {\mathrm{PI}}}}
\renewcommand{\PII}{{\mathrm{PII}}}
\newcommand{\PIII}{{\mathrm{PIII}}}
\newcommand{\PIV}{{\mathrm{PIV}}}
\newcommand{\PV}{{\mathrm{PV}}}
\newcommand{\PVI}{{\mathrm{PVI}}}

\newcommand{\FN}{{\mathrm{FN}}}

\newcommand{\Dsix}{D_6}
\newcommand{\Dseven}{D_7}
\newcommand{\Deight}{D_8}
\newcommand{\pdeg}{\mathrm{deg}}

\newcommand{\PIIFN}{{\PII(\FN)}}
\newcommand{\PIIIDsix}{{\PIII(\Dsix)}}
\newcommand{\PIIIDseven}{{\PIII(\Dseven)}}
\newcommand{\PIIIDeight}{{\PIII(\Deight)}}
\newcommand{\PVdeg}{{\PV(\pdeg)}}

\newcommand{\un}{\mathrm{un}}

\DeclareMathOperator{\Coh}{Coh}
\DeclareMathOperator{\Sh}{Sh}
\DeclareMathOperator{\Loc}{Loc}

\newcommand{\cD}{\mathcal{D}}
\newcommand{\cH}{\mathcal{H}}
\newcommand{\cU}{\mathcal{U}}
\newcommand{\cX}{\mathcal{X}}

\newcommand{\cB}{\mathcal{B}} 

\newcommand{\gen}{\mathrm{gen}}

%% file: sections/introduction.tex

\section{Introduction}

Mirror symmetry is paradigmatically about the study of dual torus fibrations \cite{SYZ}, and   Hitchin fibrations of moduli of Higgs bundles on curves have long been recognized as a natural example \cite{hausel-thaddeus}, where mirror symmetry is expected to moreover be synonymous with  geometric Langlands \cite{kapustin-witten} -- a subject which has seen much recent progress \cite{gaitsgory-raskin-proof-I, ABCCFGKRR-proof-II}. 

Nevertheless, homological mirror symmetry for these moduli spaces -- in the original sense of matching a Fukaya category of a symplectic manifold with a category of coherent sheaves on an algebraic variety \cite{kontsevich-homological} -- has remained elusive.
In fact, we know of at most one instance for which a complete mirror theorem was previously known: the moduli of rank two local systems on the projective line with a single irregular singularity with trefoil Stokes Legendrian -- i.e. the Painlev\'e I character variety -- is self-mirror: its Fukaya category is isomorphic to its category of coherent sheaves.\footnote{Even in this case, it is not entirely trivial to extract a proof from the literature.  In \cite{casals-murphy}, one finds a determination that the A-side space is obtained by attaching a handle to the 4-ball along a Legendrian trefoil.  The result would then follow from the surgery formula for the Fukaya category \cite{BEE}, were it known that the Legendrian DGA of said trefoil had no higher cohomology.  Said vanishing has no direct proof in the literature.  Experts know various ways to complete the argument; we will give one below.}

Let us clarify a confusing point: moduli of Higgs bundles are hyperk\"ahler manifolds, and  carry many symplectic forms.  In the above assertion, and the remainder of the article, we will use the real part of the holomorphic symplectic form on the Higgs moduli.  In particular, the symplectic form is exact, and the Hitchin fibers are Lagrangian.
It is known that this symplectic form is carried, by the nonabelian Hodge diffeomorphism \cite{hitchin-self-duality, donaldson-twisted, simpson-noncompact, biquard-boalch}, to a K\"ahler form for the corresponding character variety.
That is, the K\"ahler structures relevant to our desired mirror symmetry are those of the character variety, though the mirror SYZ fibrations are most readily described via the algebraic geometry of Hitchin's integrable system.  As we will not explicitly use said fibrations here, we will stick to the character variety description henceforth. 

Here we will study the lowest -- two complex -- dimensional examples: the moduli spaces of rank two irregular local systems on the projective line which are the (Betti realizations of) Okamoto's spaces of initial conditions \cite{okamoto-initial} for the Painlev\'e equations, which we will refer to as Painlev\'e character varieties. These can also be understood as parameterizing sheaves on $\P^1$ with microsupport on the Legendrians whose front projections are depicted in \zcref{fig:stokes_legendrians}.

\begin{figure}
    \centering
    \def\svgwidth{0.6\textwidth}
    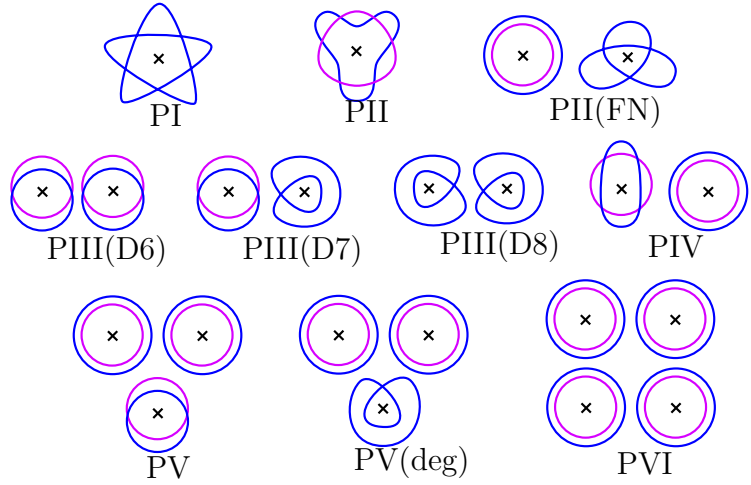
    \caption{Stokes Legendrians for all Painlevé types.}
    \label{fig:stokes_legendrians}
\end{figure}

More precisely, the relevant moduli of (irregular) local systems come in families parameterized by the (formal) monodromies at punctures; only after fixing said monodromies does the space become two dimensional. 

In fact, the Painlev\'e character varieties can be exhibited as certain explicit families of affine cubic surfaces.  In the most classical case -- the rank two characters of the fundamental group of the four punctured sphere -- this identification dates at least to the 1890s \cite{fricke, vogt}.  The other cases can be found in \cite{vanderput-saito}. 

We denote the set of Painlevé types as:
\[
\{\PI,\PII,\PIIFN,\PIIIDsix,\PIIIDseven,\PIIIDeight,\PIV,\PV,\PVdeg,\PVI\}.
\]

In this article we prove:

\begin{theorem} \label{thm: mirror}
    The Fukaya category of a generic Painlev\'e character variety 
    is equivalent to the derived category of coherent sheaves of the minimal resolution of the most singular Painlev\'e character variety 
    of the same irregular type.
\end{theorem}

In \zcref{table:pain_params}, we specify the explicit parameters which cut out the most singular member of each respective family. One can check that imposing this choice of parameters is equivalent to demanding that the microlocal monodromy along the Stokes Legendrian be trivial.

The proof combines two brute calculations -- one algebraic, one symplectic -- for each Painlev\'e type -- with some standard methods of homological mirror symmetry (the Fukaya/microsheaf correspondence  \cite{GPS3} and the microsheaf model of toric mirror symmetry  \cite{FLTZ}).

\begin{figure}
    \centering
    \def\svgwidth{0.65\textwidth}
    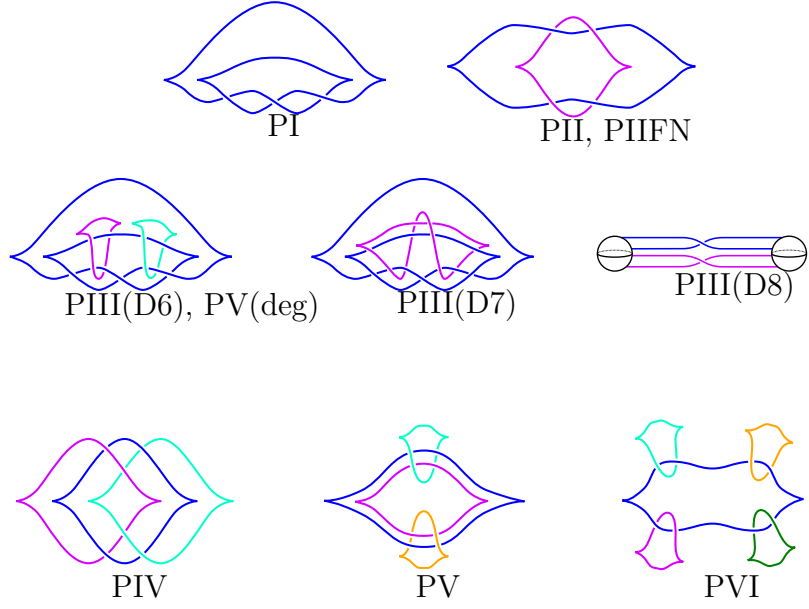
    \caption{Painlevé character varieties as Legendrian handle attachments along links in $J^1 \R \subset \partial B^4$.}
    \label{fig:painleve_handlebodies}
\end{figure}

\begin{figure}
    \centering
    \def\svgwidth{0.65\textwidth}
    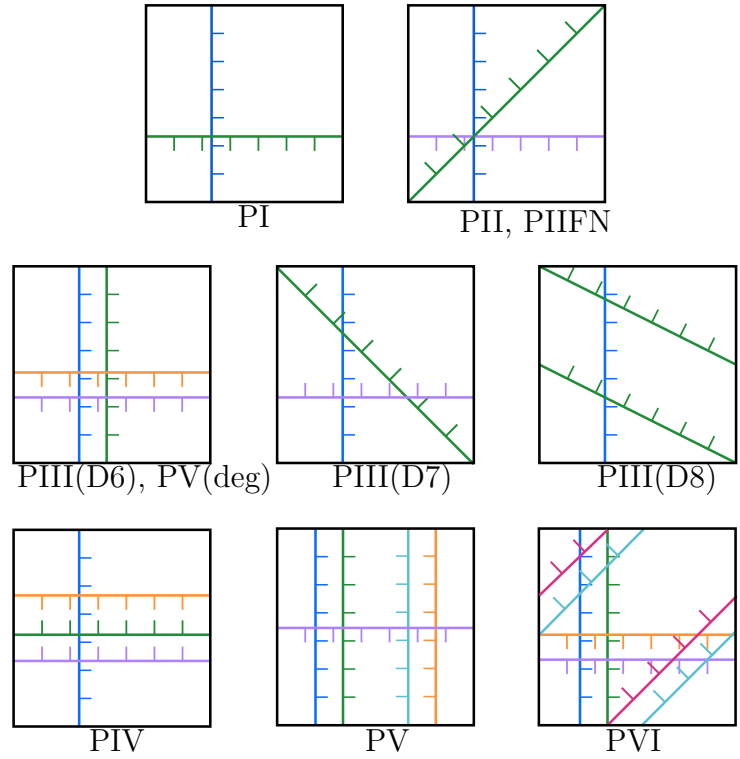
    \caption{Painlevé character varieties as Legendrian handle attachments along links in $S^* T^2 = \partial T^* T^2$.}
    \label{fig:torus_skeleta}
\end{figure}

The  symplectic topology of the general member of each family was characterized  in \cite{beimler-olsen}, where a Weinstein handle presentation was computed from the aforementioned cubic equations by applying the explicit Lefschetz fibration methods of \cite{casals-murphy}.  One presentation of the result was given in  \zcref{fig:painleve_handlebodies}.  

In \cite{beimler-olsen},  another equivalent presentation for the same Weinstein manifolds was derived: attach handles along the Stokes Legendrians (shown in \zcref{fig:stokes_legendrians}) of the corresponding irregular type of linear differential equation. This presentation indicates that a Lagrangian skeleton of the resulting Weinstein manifold contains, as an open subset, the conical locus in the cotangent bundle of $\P^1$ which gives the microlocal characterization of the irregular local systems of the given type.  The handle attachments impose the condition that the microlocal monodromies must extend over a disk, hence be trivial.  This gives, for formal reasons, a map from the Fukaya category of said Weinstein manifold to the category of coherent sheaves on the moduli of said local systems, suggesting some relation along the lines of Theorem \ref{thm: mirror}.  However, we don't know how to directly analyze this map.

Instead, we will obtain, by handle calculus, a  presentation of these spaces of the form:   attach 2-handles to the cotangent bundle of the torus along Legendrian lifts of embedded curves.  

Let us explain one reason to expect such a presentation to exist.  As shown in \cite{STWZ}, the Stokes Legendrians of irregular local systems admit Lagrangian fillings.  The Lagrangians here would be punctured tori; attaching  handles along Stokes Legendrians fills the punctures.  Thus our  spaces  contain an exact Lagrangian torus.  Inspection of the \cite{STWZ} construction  suggests that the desired 4-manifold is recovered by attaching disks to said torus to reconstruct the original zero section of the $T^* \mathbb{P}^1$; this view is substantiated by the fact that the corresponding cluster structures arise geometrically in such a manner \cite{STW, casals-gao}.  But rather than attempt to make this construction precise, here we just directly manipulate the handle presentation.  The result is the following:

\begin{theorem}\label{thm:torus_skeleta}
 \zcref{fig:painleve_handlebodies} and \zcref{fig:torus_skeleta} present equivalent Weinstein manifolds.
\end{theorem}

\zcref{thm:torus_skeleta} of PI has long been known to experts\footnote{In the context of the expected comparison between \cite{STWZ} and \cite{STW}, the last-named author learned how to do this calculation from some combination of Casals, Murphy, and Starkston circa 2017.} and appears e.g. in the proof of \cite[Thm. 10.1]{acu_complements}.  The 
result for PII is asserted (without proof) in \cite[Fig. 26 and 27]{acu_intro}.  The remaining cases are new here, but are proven similarly.


By \cite{GPS3}, we may compute Fukaya categories of Weinstein manifolds via  microsheaves  \cite{kashiwara-schapira, shende-microlocal, nadler-shende} on their Lagrangian skeleta.  
Here we have a skeleton obtained by attaching handles to conormal lifts of curves on tori.
The corresponding category can be computed by a gluing calculation, with the most complicated factor being that given by the part already contained inside the $T^* T^2$; this, conveniently, is a (very) special case of the skeleta known to be mirror to toric stacks \cite{FLTZ, FLTZ-stack, Kuwagaki-CCC}.  We give the fans characterizing the corresponding toric stacks in \zcref{fig:stacky-toric-models}.

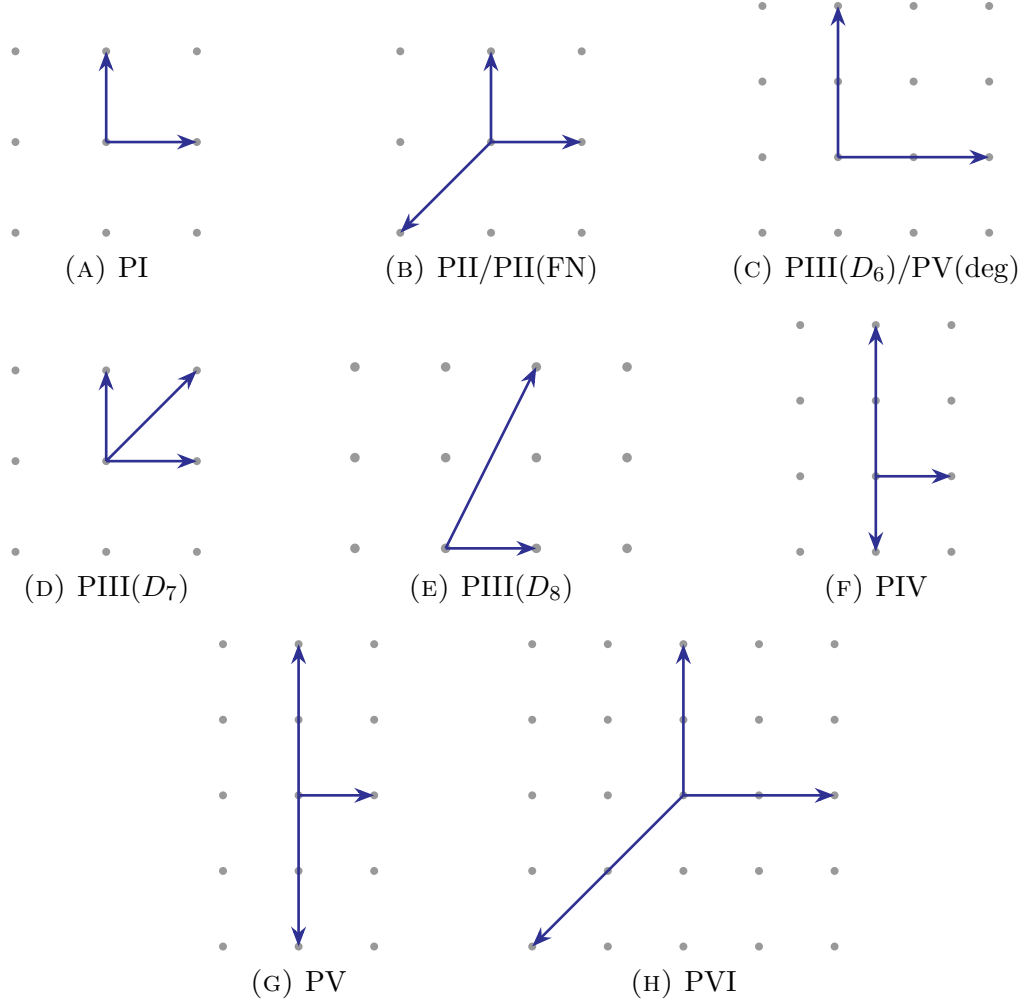
\begin{figure}
    \centering
    \begin{subfigure}[b]{.3 \textwidth}
        \centering
        
        \begin{tikzpicture}[line width = 1pt, Blue, scale = 1.2]
        \foreach \x in {-1,...,1}
            \foreach \y in {-1,...,1}
                \node[circle, fill=gray!80, inner sep=0pt, minimum size=3pt] at (\x,\y) {};
                
            \draw[-Stealth] (0, 0) -- (0, 1);
            \draw[-Stealth] (0, 0) -- (1, 0); 
        \end{tikzpicture}
        \caption{ \(\PI\)}
    \end{subfigure}
    \begin{subfigure}{.3 \textwidth}
        \centering
        \begin{tikzpicture}[Blue, line width = 1pt, scale = 1.2]
        \foreach \x in {-1,...,1}
            \foreach \y in {-1,...,1}
                \node[circle, fill=gray!80, inner sep=0pt, minimum size=3pt] at (\x,\y) {};
            \draw[-Stealth] (0, 0) -- (0, 1);
            \draw[-Stealth] (0, 0) -- (1, 0); 
            \draw[-Stealth] (0, 0) -- (-1, -1);
        \end{tikzpicture}
        \caption{ \(\PII / \PIIFN\)}
    \end{subfigure}
    \begin{subfigure}{.3 \textwidth}
        \centering
        \begin{tikzpicture}[line width = 1pt, scale = 1, Blue]
        \foreach \x in {-1,...,2}
            \foreach \y in {-1,...,2}
                \node[circle, fill=gray!80, inner sep=0pt, minimum size=3pt] at (\x,\y) {};
            \draw[-Stealth] (0, 0) -- (0, 2);
            \draw[-Stealth] (0, 0) -- (2, 0);
        \end{tikzpicture}
        \caption{ \(\PIIIDsix / \PVdeg\)}
    \end{subfigure}

    \vspace{1em}
     \begin{subfigure}{.3 \textwidth}
        \centering
        \begin{tikzpicture}[line width = 1pt, scale = 1.2, Blue]
        \foreach \x in {-1,...,1}
            \foreach \y in {-1,...,1}
                \node[circle, fill=gray!80, inner sep=0pt, minimum size=3pt] at (\x,\y) {};
            \draw[-Stealth] (0, 0) -- (0, 1);
            \draw[-Stealth] (0, 0) -- (1, 0);
            \draw[-Stealth] (0, 0) -- (1, 1); 
        \end{tikzpicture}
        \caption{ \(\PIIIDseven\)}
    \end{subfigure}
    \begin{subfigure}{.3 \textwidth}
        \centering
        \begin{tikzpicture}[line width = 1pt, scale = 1.2, Blue]
        \foreach \x in {-1,...,2}
            \foreach \y in {0,...,2}
                \fill[gray!80] (\x,\y) circle (1.5pt);
            \draw[-Stealth] (0, 0) -- (1, 0);
            \draw[-Stealth] (0, 0) -- (1, 2); 
        \end{tikzpicture}
        \caption{ \(\PIIIDeight\)}
    \end{subfigure}
    \begin{subfigure}{.3 \textwidth}
        \centering
        \begin{tikzpicture}[line width = 1pt, scale = 1, Blue]
        \foreach \x in {-1,...,1}
            \foreach \y in {-1,...,2}
                \fill[gray!80] (\x,\y) circle (1.5pt);
            \draw[-Stealth] (0, 0) -- (0, 2);
            \draw[-Stealth] (0, 0) -- (1, 0); 
            \draw[-Stealth] (0, 0) -- (0, -1); 
            
        \end{tikzpicture}
        \caption{ \(\PIV\)}
    \end{subfigure}

    \vspace{1em}
    \begin{subfigure}{.3 \textwidth}
        \centering
        \begin{tikzpicture}[line width = 1pt, scale = 1, Blue]
        \foreach \x in {-1,...,1}
            \foreach \y in {-2,...,2}
                \fill[gray!80] (\x,\y) circle (1.5pt);
            \draw[-Stealth] (0, 0) -- (0, 2);
            \draw[-Stealth] (0, 0) -- (1, 0); 
            \draw[-Stealth] (0, 0) -- (0, -2);
            
        \end{tikzpicture}
        \caption{ \(\PV\)}
    \end{subfigure}
    \begin{subfigure}{.3 \textwidth}
        \centering
        \begin{tikzpicture}[line width = 1pt, scale = 1, Blue]
        \foreach \x in {-2,...,2}
            \foreach \y in {-2,...,2}
                \fill[gray!80] (\x,\y) circle (1.5pt);
            \draw[-Stealth] (0, 0) -- (0, 2);
            \draw[-Stealth] (0, 0) -- (2, 0); 
            \draw[-Stealth] (0, 0) -- (-2, -2);
            
        \end{tikzpicture}
        \caption{ \(\PVI\)}
    \end{subfigure}
    \caption{The fans of the toric stacks mirror, under \cite{FLTZ, FLTZ-stack}, to the skeleta of \zcref{fig:torus_skeleta}. Note that the fans do not contain \(2\)-dimensional cones.  }
    \label{fig:stacky-toric-models}
\end{figure}

The gluing calculation we need is essentially the main result of \cite{gammage-le}, though strictly speaking we  need a slight generalization, given in \zcref{thm:HMS-for-U-M}.  We will also trade the resulting stacky objects for resolutions of their coarse moduli (which in these cases  have equivalent derived categories \cite{Kapranov-Vasserot}).  
We deduce that each Fukaya category in question is equivalent to the  category of coherent sheaves on the complement of the strict transform of the toric boundary divisor in a certain explicit (nontoric) blowup of a toric surface, illustrated in \zcref{fig:B-side toric models}.

\begin{figure}

\centering

\begin{subfigure}[b]{.3 \textwidth}
    \centering
\begin{tikzpicture}[line width = 1pt, scale = .8]
    \draw[blue] (3, 0) -- (-1, 0) -- (-1,4);
    \draw (1, -.4) .. controls (.9, .3) and (1.1, .8) .. (1.2, 1);
    \draw (-1.4,  2) .. controls (-.8, 1.9) and (-.4, 2) .. (0, 2.2); 

\end{tikzpicture}
\caption{\(\PI\)}
\end{subfigure}
\begin{subfigure}[b]{.3\textwidth}
    \centering
\begin{tikzpicture}[line width = 1pt, scale = .7]
    \draw[blue] (4, 0) -- (-1, 0) -- (-1,5) -- (4, 0);
    \draw (1.5, -.4) .. controls (1.4, .3) and (1.3, .8) .. (1.2, 1);
    \draw (-1.4,  2.5) .. controls (-1, 2.3) and (-.4, 2.2) .. (0, 2.2); 
    \draw (1.8, 2.8) ..controls (1.5, 2.6) and (1.1, 2.2) .. (.9, 1.9); 
    \draw[rotate=-45, red] (-.7, 1.4) ellipse (1.2 and 1);
    
\end{tikzpicture}

\caption{\(\PII / \PIIFN\)}
\end{subfigure}
\begin{subfigure}[b]{.3\textwidth}
\centering
\begin{tikzpicture}[yscale = .35, xscale = .35, line width = 1pt]
    \draw[blue] (10, 0) -- (0,0) -- (0, 10); 

    \draw (0,0) (5, -1) arc (-10:20:5);
    \draw[red] (0, 2) ++(-17:5) arc (-20:10:5);
    
    \draw (0,0) ++(95: 5) arc (100:70:5);
    \draw[red] (2, 0) ++(105:5) arc (105:75:5.5);
\end{tikzpicture}
    \caption{\(\PIIIDsix/\PVdeg\)}
\end{subfigure}

\begin{subfigure}[b]{.3\textwidth}
\centering
\begin{tikzpicture}[yscale = .35, xscale = .35, line width = 1pt]
    \draw[blue] (10, 0) -- (3,0) -- (0, 3) -- (0, 10); 

    \begin{scope}[shift = {(2,0)}]
    \draw (0,0) ++ (-5: 5) arc (-5:25:5);
    \end{scope}
    
    \begin{scope}[shift = {(0,2)}]
    \draw (0,0) ++(95: 5) arc (95:70:5);
    \end{scope}

    \begin{scope}[shift = {(5,-2)}, rotate = 45]
    \draw (0,0) ++(95: 5) arc (95:70:5);
    \end{scope}

\end{tikzpicture}
\caption{\(\PIIIDseven\)}
\end{subfigure}
\begin{subfigure}[b]{.3 \textwidth}
    \centering
    \begin{tikzpicture}[scale = .24, line width = 1pt]
    \draw[blue] (-10, 5) -- (0, 0) -- (10, 0);

    \draw (0,0) (5, -1) arc (-10:25:6);

    \begin{scope}[rotate = 140]
        \draw (0,0) (5, -1) arc (-5:25:6);

    \end{scope}
    \end{tikzpicture}
    \caption{\(\PIIIDeight\)}
\end{subfigure}
\begin{subfigure}[b]{.3\textwidth}
    \centering
    \begin{tikzpicture}[scale = .3, line width = 1pt]
    \draw[blue] (10, 10) -- (0, 10) -- (0,0) -- (10, 0);
    
    \draw (0,1) ++ (-20: 5) arc (-20:20:5);
    
    \draw (0, 10) ++ (5: 5) arc (5: -20: 6);
    \draw[red] (0, 7) ++ (25: 5) arc (30: -5: 5);

    \draw[red] (5, 4.3) ellipse (1.2 and 2.3);
    
    \draw (0, 0) ++ (95: 5) arc (95: 65:5);
    \end{tikzpicture}
    \caption{\(\PIV\)}
\end{subfigure}

\begin{subfigure}[b]{.45\textwidth}
    \centering
    
    \begin{tikzpicture}[scale = .4, line width = 1pt]
    \draw[blue] (10, 10) -- (0, 10) -- (0,0) -- (10, 0) ;
    
    \draw (0,0) ++ (-5: 5) arc (-5:25:5);
    \draw[red] (0, 2) ++ (-10: 4.7) arc (-20:20:4);
    
    \draw (0, 10) ++ (5: 5) arc (5: -25:5);
    \draw[red] (0, 7.8) ++ (15: 4.8) arc (20: -15: 4.5);
    
    \draw (0, 0) ++ (95: 5) arc (95: 65:5);
     \draw[red] (4.8, 5) ellipse (1.2 and 1.9);
    \end{tikzpicture}

    \caption{\(\PV\)}
\end{subfigure}
\begin{subfigure}[b]{.45\textwidth}
    \centering
    \begin{tikzpicture}[scale = .28, line width = 1pt]
        \draw[blue] (-2, 0) -- (14, 0) -- (-2, 16) -- cycle;
        \draw (2.5, -1) ++ (10: 3) arc (10: 70: 3);
        \draw[red] (10, 3) ++ (-5: -6.2) arc (-5: 25: -5); 
        
        \draw (-3, 3.5) ++ (90: 3) arc (90: 35: 3 ); 
        \draw[red] (-.1, 2.7) ++ (110: 3) arc (110: 60: 3);

        \draw (10, 5.5) ++ (140: 5) arc (140: 165: 5);
        \draw[red] (7.6, 4) ++ (120: 4) arc (120: 155: 4);
        \draw[rotate=0, red] (2.8, 4.7) ellipse (2.1 and 2.1);
    \end{tikzpicture}
    \caption{\(\PVI\)}

\end{subfigure}

\caption{The toric models on the B-side.  The {\color{blue}{blue}} lines are the strict transforms of the original toric boundary, and the arcs are the exceptional divisors of blowups.  The desired surface is obtained by deleting the strict transform of the toric boundary.  Note that the \red{red} curves are \(\bP^1\)'s which survive in the resulting open surface (and so the cases with red arcs are not affine).}
  \label{fig:B-side toric models}
\end{figure}
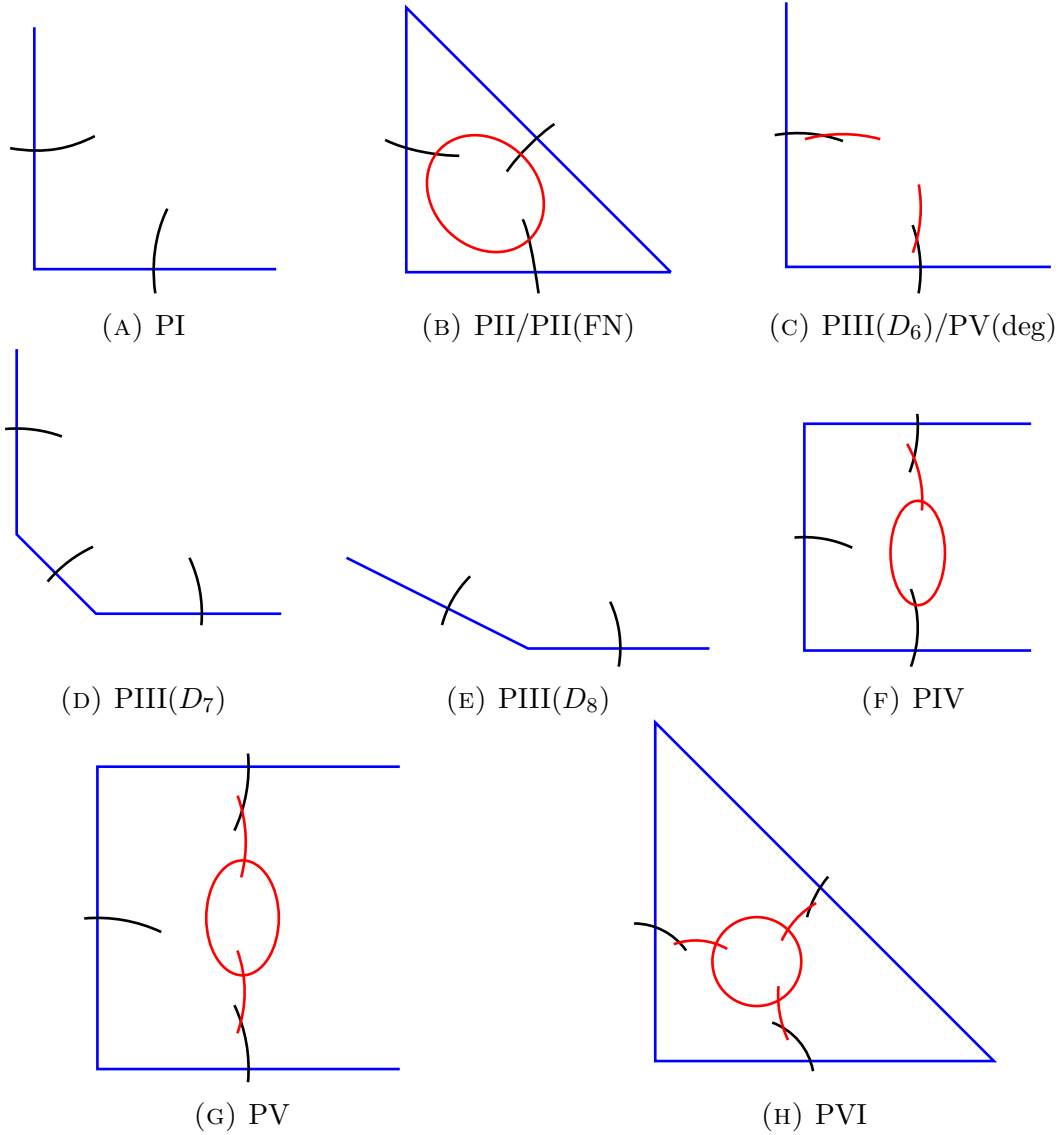

To complete the proof of our asserted mirror symmetry, we prove: 

\begin{theorem}[\zcref{thm: B precise}]
\label{thm: B} 
    For each Painlev\'e type, the surface which is mirror to the handle attachment depicted in \zcref{fig:torus_skeleta} -- i.e., the complement of the strict transform of the toric boundary divisor in a certain  (nontoric) blowup of a certain toric surface -- is isomorphic to the minimal resolution of the corresponding unipotent Painlev\'e character variety. 
\end{theorem}

The main step in the proof of \zcref{thm: B} is determining explicitly the rings of global functions on the mirror surfaces.   Plausibly, this could be done  using the ideas of \cite{sakai}. 

Instead, we asked ChatGPT, which was able to provide an explicit answer in coordinates (which we then verified).
More precisely: we knew \zcref{thm:torus_skeleta} for $\PI$ and $\PII$, and thus expected \zcref{thm: B}; we verified it after obtaining the regular functions from ChatGPT.  By analyzing the topology and singularities of the Painlev\'e character varieties, we were able to guess toric models of the form: blow up these points on that toric surface.  From this description, ChatGPT could determine regular functions, which we verified to be correct and to agree with the explicit formulas of \cite{vanderput-saito}.  Finally, these B-side calculations suggested (by mirror symmetry) the formulation of \zcref{thm:torus_skeleta}, which we then proved by handle calculus.


\clearpage

\begin{remark}
    Mirror symmetry for character varieties is closely related to the `Betti' form of the geometric Langlands correspondence \cite{nadler-benzvi}; indeed, it is expected to be a microlocalization thereof; see e.g. discussions in the introductions of \cite{shende-higgs, nadler-shende-higgs}.  
    However, we do not know how to deduce \zcref{thm: mirror} from known results in Betti Langlands.  There are at least three reasons for this.  The first is that we are not aware of a precise formulation of the geometric Langlands equivalence in the presence of wild ramification, though some directions can be found in \cite{witten-wild}.  
    In the tamely ramified case -- which overlaps with the Painlev\'e series only in $\PVI$ -- such a conjectural formulation is known, and it is possible that Betti Langlands for $\PVI$ can be extracted from the literature.  But, second, even here the authors do not understand how Langlands predicts that passing to generic  monodromy on one side matches passing to unipotent  monodromy on the other.  Finally, even in the unramified case --- where the Betti correspondence was shown, as a corollary of the more famous de Rham version, in \cite{gaitsgory-raskin-proof-I, ABCCFGKRR-proof-II} --  it is not presently known how to derive the desired mirror symmetry results from the  geometric Langlands equivalence; see e.g. \cite[Question 1.3]{nadler-shende-higgs} for an indication of our ignorance. 
\end{remark}

\vspace{2mm}
{\bf Acknowledgements.}  The authors are  supported by Villum Fonden Villum Investigator grant 37814.

%% file: images/stokes_legendrians.pdf_tex
\begingroup%
  \makeatletter%
  \providecommand\color[2][]{%
    \errmessage{(Inkscape) Color is used for the text in Inkscape, but the package 'color.sty' is not loaded}%
    \renewcommand\color[2][]{}%
  }%
  \providecommand\transparent[1]{%
    \errmessage{(Inkscape) Transparency is used (non-zero) for the text in Inkscape, but the package 'transparent.sty' is not loaded}%
    \renewcommand\transparent[1]{}%
  }%
  \providecommand\rotatebox[2]{#2}%
  \newcommand*\fsize{\dimexpr\f@size pt\relax}%
  \newcommand*\lineheight[1]{\fontsize{\fsize}{#1\fsize}\selectfont}%
  \ifx\svgwidth\undefined%
    \setlength{\unitlength}{390.55181068bp}%
    \ifx\svgscale\undefined%
      \relax%
    \else%
      \setlength{\unitlength}{\unitlength * \real{\svgscale}}%
    \fi%
  \else%
    \setlength{\unitlength}{\svgwidth}%
  \fi%
  \global\let\svgwidth\undefined%
  \global\let\svgscale\undefined%
  \makeatother%
  \begin{picture}(1,0.64908767)%
    \lineheight{1}%
    \setlength\tabcolsep{0pt}%
    \put(0,0){\includegraphics[width=\unitlength,page=1]{stokes_legendrians.pdf}}%
    \put(0.19360898,0.47628519){\color[rgb]{0,0,0}\makebox(0,0)[lt]{\lineheight{1.25}\smash{\begin{tabular}[t]{l}PI\end{tabular}}}}%
    \put(0,0){\includegraphics[width=\unitlength,page=2]{stokes_legendrians.pdf}}%
    \put(0.45398303,0.48240709){\color[rgb]{0,0,0}\makebox(0,0)[lt]{\lineheight{1.25}\smash{\begin{tabular}[t]{l}PII\\\end{tabular}}}}%
    \put(0,0){\includegraphics[width=\unitlength,page=3]{stokes_legendrians.pdf}}%
    \put(0.7274159,0.48540579){\color[rgb]{0,0,0}\makebox(0,0)[lt]{\lineheight{1.25}\smash{\begin{tabular}[t]{l}PII(FN)\\\end{tabular}}}}%
    \put(0,0){\includegraphics[width=\unitlength,page=4]{stokes_legendrians.pdf}}%
    \put(0.05662207,0.29936745){\color[rgb]{0,0,0}\makebox(0,0)[lt]{\lineheight{1.25}\smash{\begin{tabular}[t]{l}PIII(D6)\\\\\end{tabular}}}}%
    \put(0,0){\includegraphics[width=\unitlength,page=5]{stokes_legendrians.pdf}}%
    \put(0.31750271,0.29962564){\color[rgb]{0,0,0}\makebox(0,0)[lt]{\lineheight{1.25}\smash{\begin{tabular}[t]{l}PIII(D7)\\\\\end{tabular}}}}%
    \put(0,0){\includegraphics[width=\unitlength,page=6]{stokes_legendrians.pdf}}%
    \put(0.58145846,0.30591028){\color[rgb]{0,0,0}\makebox(0,0)[lt]{\lineheight{1.25}\smash{\begin{tabular}[t]{l}PIII(D8)\\\end{tabular}}}}%
    \put(0,0){\includegraphics[width=\unitlength,page=7]{stokes_legendrians.pdf}}%
    \put(0.8590786,0.30120397){\color[rgb]{0,0,0}\makebox(0,0)[lt]{\lineheight{1.25}\smash{\begin{tabular}[t]{l}PIV\end{tabular}}}}%
    \put(0,0){\includegraphics[width=\unitlength,page=8]{stokes_legendrians.pdf}}%
    \put(0.18887195,0.00719901){\color[rgb]{0,0,0}\makebox(0,0)[lt]{\lineheight{1.25}\smash{\begin{tabular}[t]{l}PV\end{tabular}}}}%
    \put(0,0){\includegraphics[width=\unitlength,page=9]{stokes_legendrians.pdf}}%
    \put(0.46658412,0.01796028){\color[rgb]{0,0,0}\makebox(0,0)[lt]{\lineheight{1.25}\smash{\begin{tabular}[t]{l}PV(deg)\\\end{tabular}}}}%
    \put(0,0){\includegraphics[width=\unitlength,page=10]{stokes_legendrians.pdf}}%
    \put(0.81762017,0.01212932){\color[rgb]{0,0,0}\makebox(0,0)[lt]{\lineheight{1.25}\smash{\begin{tabular}[t]{l}PVI\end{tabular}}}}%
  \end{picture}%
\endgroup%

%% file: images/s3_painleve_links.pdf_tex
\begingroup%
  \makeatletter%
  \providecommand\color[2][]{%
    \errmessage{(Inkscape) Color is used for the text in Inkscape, but the package 'color.sty' is not loaded}%
    \renewcommand\color[2][]{}%
  }%
  \providecommand\transparent[1]{%
    \errmessage{(Inkscape) Transparency is used (non-zero) for the text in Inkscape, but the package 'transparent.sty' is not loaded}%
    \renewcommand\transparent[1]{}%
  }%
  \providecommand\rotatebox[2]{#2}%
  \newcommand*\fsize{\dimexpr\f@size pt\relax}%
  \newcommand*\lineheight[1]{\fontsize{\fsize}{#1\fsize}\selectfont}%
  \ifx\svgwidth\undefined%
    \setlength{\unitlength}{337.30014831bp}%
    \ifx\svgscale\undefined%
      \relax%
    \else%
      \setlength{\unitlength}{\unitlength * \real{\svgscale}}%
    \fi%
  \else%
    \setlength{\unitlength}{\svgwidth}%
  \fi%
  \global\let\svgwidth\undefined%
  \global\let\svgscale\undefined%
  \makeatother%
  \begin{picture}(1,0.75205057)%
    \lineheight{1}%
    \setlength\tabcolsep{0pt}%
    \put(0,0){\includegraphics[width=\unitlength,page=1]{s3_painleve_links.pdf}}%
    \put(0.66247737,0.57065267){\color[rgb]{0,0,0}\makebox(0,0)[lt]{\lineheight{1.25}\smash{\begin{tabular}[t]{l}PII, PIIFN\end{tabular}}}}%
    \put(0,0){\includegraphics[width=\unitlength,page=2]{s3_painleve_links.pdf}}%
    \put(0.07662161,0.35681732){\color[rgb]{0,0,0}\makebox(0,0)[lt]{\lineheight{1.25}\smash{\begin{tabular}[t]{l}PIII(D6), PV(deg)\end{tabular}}}}%
    \put(0,0){\includegraphics[width=\unitlength,page=3]{s3_painleve_links.pdf}}%
    \put(0.82854968,0.38072757){\color[rgb]{0,0,0}\makebox(0,0)[lt]{\lineheight{1.25}\smash{\begin{tabular}[t]{l}PIII(D8)\end{tabular}}}}%
    \put(0,0){\includegraphics[width=\unitlength,page=4]{s3_painleve_links.pdf}}%
    \put(0.48633005,0.35897366){\color[rgb]{0,0,0}\makebox(0,0)[lt]{\lineheight{1.25}\smash{\begin{tabular}[t]{l}PIII(D7)\end{tabular}}}}%
    \put(0,0){\includegraphics[width=\unitlength,page=5]{s3_painleve_links.pdf}}%
    \put(0.1345779,0.00833718){\color[rgb]{0,0,0}\makebox(0,0)[lt]{\lineheight{1.25}\smash{\begin{tabular}[t]{l}PIV\end{tabular}}}}%
    \put(0,0){\includegraphics[width=\unitlength,page=6]{s3_painleve_links.pdf}}%
    \put(0.5095329,0.00833718){\color[rgb]{0,0,0}\makebox(0,0)[lt]{\lineheight{1.25}\smash{\begin{tabular}[t]{l}PV\end{tabular}}}}%
    \put(0,0){\includegraphics[width=\unitlength,page=7]{s3_painleve_links.pdf}}%
    \put(0.86458617,0.00833716){\color[rgb]{0,0,0}\makebox(0,0)[lt]{\lineheight{1.25}\smash{\begin{tabular}[t]{l}PVI\end{tabular}}}}%
    \put(0,0){\includegraphics[width=\unitlength,page=8]{s3_painleve_links.pdf}}%
    \put(0.32683494,0.5775546){\color[rgb]{0,0,0}\makebox(0,0)[lt]{\lineheight{1.25}\smash{\begin{tabular}[t]{l}PI\end{tabular}}}}%
  \end{picture}%
\endgroup%

%% file: images/torus_skeleta.pdf_tex
\begingroup%
  \makeatletter%
  \providecommand\color[2][]{%
    \errmessage{(Inkscape) Color is used for the text in Inkscape, but the package 'color.sty' is not loaded}%
    \renewcommand\color[2][]{}%
  }%
  \providecommand\transparent[1]{%
    \errmessage{(Inkscape) Transparency is used (non-zero) for the text in Inkscape, but the package 'transparent.sty' is not loaded}%
    \renewcommand\transparent[1]{}%
  }%
  \providecommand\rotatebox[2]{#2}%
  \newcommand*\fsize{\dimexpr\f@size pt\relax}%
  \newcommand*\lineheight[1]{\fontsize{\fsize}{#1\fsize}\selectfont}%
  \ifx\svgwidth\undefined%
    \setlength{\unitlength}{178.98235892bp}%
    \ifx\svgscale\undefined%
      \relax%
    \else%
      \setlength{\unitlength}{\unitlength * \real{\svgscale}}%
    \fi%
  \else%
    \setlength{\unitlength}{\svgwidth}%
  \fi%
  \global\let\svgwidth\undefined%
  \global\let\svgscale\undefined%
  \makeatother%
  \begin{picture}(1,0.97063714)%
    \lineheight{1}%
    \setlength\tabcolsep{0pt}%
    \put(0,0){\includegraphics[width=\unitlength,page=1]{torus_skeleta.pdf}}%
    \put(0.36885521,0.67972463){\color[rgb]{0,0,0}\makebox(0,0)[lt]{\lineheight{1.25}\smash{\begin{tabular}[t]{l}PI\end{tabular}}}}%
    \put(0,0){\includegraphics[width=\unitlength,page=2]{torus_skeleta.pdf}}%
    \put(0.64228309,0.6729776){\color[rgb]{0,0,0}\makebox(0,0)[lt]{\lineheight{1.25}\smash{\begin{tabular}[t]{l}PII, PIIFN\end{tabular}}}}%
    \put(0,0){\includegraphics[width=\unitlength,page=3]{torus_skeleta.pdf}}%
    \put(0.09727475,0.35807223){\color[rgb]{0,0,0}\makebox(0,0)[lt]{\lineheight{1.25}\smash{\begin{tabular}[t]{l}PIII(D6), PV(deg)\end{tabular}}}}%
    \put(0,0){\includegraphics[width=\unitlength,page=4]{torus_skeleta.pdf}}%
    \put(0.48554255,0.35807223){\color[rgb]{0,0,0}\makebox(0,0)[lt]{\lineheight{1.25}\smash{\begin{tabular}[t]{l}PIII(D7)\end{tabular}}}}%
    \put(0,0){\includegraphics[width=\unitlength,page=5]{torus_skeleta.pdf}}%
    \put(0.81168198,0.35807223){\color[rgb]{0,0,0}\makebox(0,0)[lt]{\lineheight{1.25}\smash{\begin{tabular}[t]{l}PIII(D8)\end{tabular}}}}%
    \put(0,0){\includegraphics[width=\unitlength,page=6]{torus_skeleta.pdf}}%
    \put(0.18565202,0.03325235){\color[rgb]{0,0,0}\makebox(0,0)[lt]{\lineheight{1.25}\smash{\begin{tabular}[t]{l}PIV\end{tabular}}}}%
    \put(0,0){\includegraphics[width=\unitlength,page=7]{torus_skeleta.pdf}}%
    \put(0.52593461,0.03325242){\color[rgb]{0,0,0}\makebox(0,0)[lt]{\lineheight{1.25}\smash{\begin{tabular}[t]{l}PV\end{tabular}}}}%
    \put(0.82244476,0.03325242){\color[rgb]{0,0,0}\makebox(0,0)[lt]{\lineheight{1.25}\smash{\begin{tabular}[t]{l}PVI\end{tabular}}}}%
    \put(0,0){\includegraphics[width=\unitlength,page=8]{torus_skeleta.pdf}}%
  \end{picture}%
\endgroup%

%% file: sections/A-side_handleslide.tex

\section{Recollections on Legendrian handle calculus}

A Weinstein manifold is by definition an exact symplectic manifold $(X, \omega = \dd \lambda)$ such that the `Liouville' vector field $Z$ defined via
$        \imath_Z \dd \lambda = \lambda$
    has the property that there exists a proper, bounded below, Morse function $\phi: X \to \R$, satisfying $
        \dd \phi(Z) \geq \delta \norm{Z}^2
    $
    for some $\delta > 0$, with respect to some choice of metric.
    A Weinstein manifold is of finite type if $\phi$ has finitely many critical points; all manifolds considered here will be of finite type.   
    

    The skeleton of a Weinstein manifold $(X, \lambda, \phi)$ is the set
    \[
        \mathrm{Skel}(X,\lambda, \phi) = \bigcap_{t > 0} Z^{-t}(X),
    \]
    that is, the union of  stable manifolds of  critical points of $\phi$ with respect to the flow of $Z$.


Some examples of Weinstein manifolds are smooth complex affine varieties and cotangent bundles of closed manifolds. We refer to the standard textbook \cite{ce_stein_weinstein} for details.

The key consequence of this definition is that the Morse function $\phi$ induces a handlebody decomposition of $X$ compatible with the symplectic structure \cite{weinstein_original}. One may thus construct Weinstein manifolds by attaching handles along increasing sublevel sets, where the attaching spheres are given by intersecting with the skeleton: suppose $x \in X$ is a critical point of $\phi$ with critical value $\phi(x) = y$. Then for small $\epsilon$, the sublevel set $X_{\phi \leq y+\epsilon}$ is obtained from $X_{\phi \leq y-\epsilon}$ by attaching a handle along the intersection of $X_{\phi = y-\epsilon}$ with the stratum of $\mathrm{Skel}(X, \lambda, \phi)$ corresponding to the stable manifold of $x$. 

We specialize to dimension $4$. Following the above, we can obtain any connected Weinstein $4$-manifold from the $4$-ball by attaching $1$- and $2$-handles along its contact boundary $S^3$. 

One can now draw the attaching spheres of these handles in a front projection to $\R^2$, and manipulate them according to the Legendrian Reidemeister moves from \zcref{fig:reidemeister_moves}.  These simply change the presentation of the isotropic attaching spheres.
More nontrivial changes of Morse function are captured by Gompf's moves from \zcref{fig:gompf_moves}, as well as Legendrian handle slides and handle cancellations \cite{gompf_stein_handlebodies}, see also \cite{ding_geiges_handle_moves}.  These are depicted in \zcref{fig:slides_cancellations}.

\begin{figure}
    \centering
    \def\svgwidth{0.4\textwidth}
    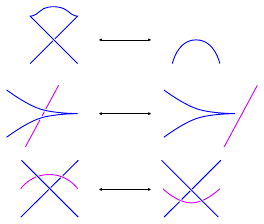
    \caption{Legendrian Reidemeister moves}
    \label{fig:reidemeister_moves}
\end{figure}

\begin{figure}
    \centering
    \def\svgwidth{0.8\textwidth}
    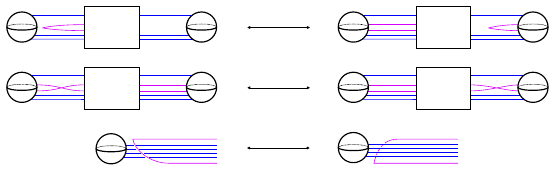
    \caption{Gompf moves}
    \label{fig:gompf_moves}
\end{figure}

\begin{figure}
    \centering
    \def\svgwidth{0.8\textwidth}
    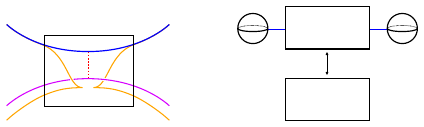
    \caption{Left: a local picture of a handle slide corresponding to the Reeb chord indicated by the red vertical line. Outside this region, the orange curve coincides with the blue curve or is a small Reeb pushoff of the magenta curve. 
    Right: a $1$- and $2$-handle in cancelling configuration. $T'$ is obtained from the Legendrian tangle $T$ by removing the blue curve.}
    \label{fig:slides_cancellations}
\end{figure}


In the special case when some sublevel set is (a Weinstein domain completing to a) cotangent bundle,  $X_{\phi < r} \cong T^*M$, there is another convenient way  of presenting Legendrian attaching spheres: as the Legendrian conormal lifts of curves drawn on the surface.\footnote{The typical front diagrams mentioned above are essentially the case $T^* {\R^2}$; the `no vertical tangencies' restriction corresponds to the fact that most authors prefer their Legendrians to lie in $\R^3 =J^1 \R^2 \subsetneq S^* \R^2 \subsetneq S^3 = \partial_\infty T^* \R^2$. 
This restriction is irrelevant for the purpose of handle attachment, and anyway has no natural analogue in general $T^*M$.  Dropping it corresponds to working in $S^* \R^2$ in this case.}  


Casals and Murphy explained how iterated Lefschetz fibration of affine varieties can be used to produce handle presentations for the underlying Weinstein manifolds \cite[Recipe 3.3]{casals-murphy}. 

We will need to translate between the two presentations mentioned above: handle attachments building from $T^* \R^2$ via Gompf diagrams, and handle attachments building from more general $T^*M$; we do this using \cite[Theorem 8.1]{acu_complements}.

\section{Toric models on the A-side.}


To prove \zcref{thm:torus_skeleta}, we apply the algorithm from \cite[Theorem 8.1]{acu_complements} to the diagrams in \zcref{fig:torus_skeleta}, and match the results with the links in \zcref{fig:painleve_handlebodies} via Legendrian handle calculus. 

In our handlebody calculus diagrams, we will often combine several simple Reidemeister moves, in particular the move from \zcref{fig:r1_cusp_move}.

\begin{figure}
    \centering
    \def\svgwidth{0.5\textwidth}
    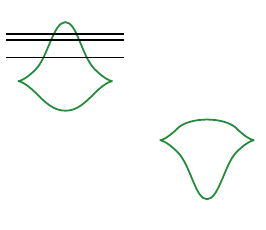
    \caption{Sliding the cusps of an unknot over horizontal strands.}
    \label{fig:r1_cusp_move}
\end{figure}

See \zcref{fig:p1} for PI, \zcref{fig:p2} for PII and PII(FN), \zcref{fig:p3d6}, \zcref{fig:PIIID6-2}, \zcref{fig:PIIID6-3} for PIII(D6) and PV(deg), \zcref{fig:p4} for PIV, \zcref{fig:p5}, \zcref{fig:p5_2} for PV, and \zcref{fig:PVI1}, \zcref{fig:PVI2}, \zcref{fig:PVI3}, \zcref{fig:PVI-continued} for PVI. 

For PVI, before the handle calculus, we perform a mutation of the curve configuration from \zcref{fig:torus_skeleta} along the blue curve, yielding \zcref{fig:p6_mutation}.
For more on mutations of curve configurations on tori, see \cite{STW}; we only need the fact that the Weinstein manifolds obtained by attaching handles along the Legendrian lifts of the curve configurations before and after mutation are Weinstein deformation equivalent.

\begin{figure}
    \centering
    \def\svgwidth{0.95\textwidth}
    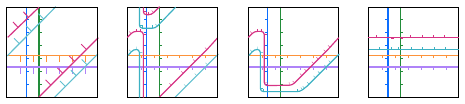
    \caption{Mutating the curve configuration for PVI along the blue curve. The computation in \zcref{fig:PVI1} uses a rotation of the last picture.}
    \label{fig:p6_mutation}
\end{figure}

\begin{figure}
    \centering
    \def\svgwidth{0.90\textwidth}
    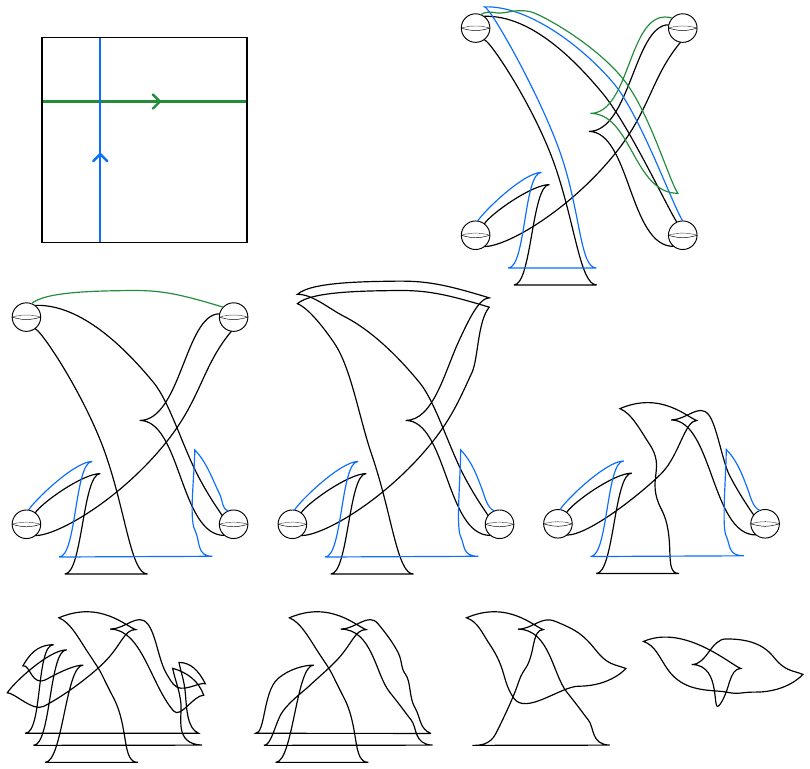
    \caption{Simplifications for PI.}
    
    \label{fig:p1}
\end{figure}

\begin{figure}[htbp]
    \centering
    \def\svgwidth{0.95\textwidth}
    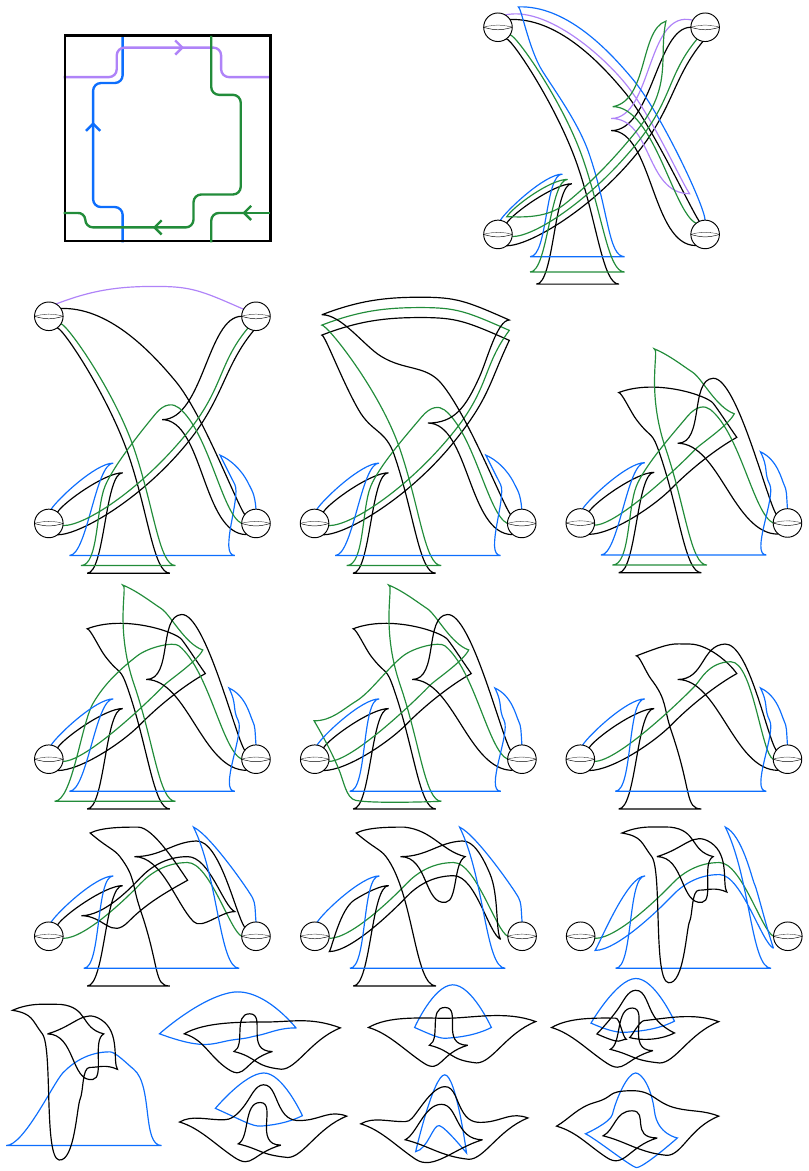
    \caption{Simplifications for PII and PII(FN). Undoing the evident Reidemeister 1 move in the last figure yields the $(2,4)$-torus link.  }
    \label{fig:p2}
\end{figure}

\begin{figure}[htbp]
    \centering
    \def\svgwidth{0.95\textwidth}
    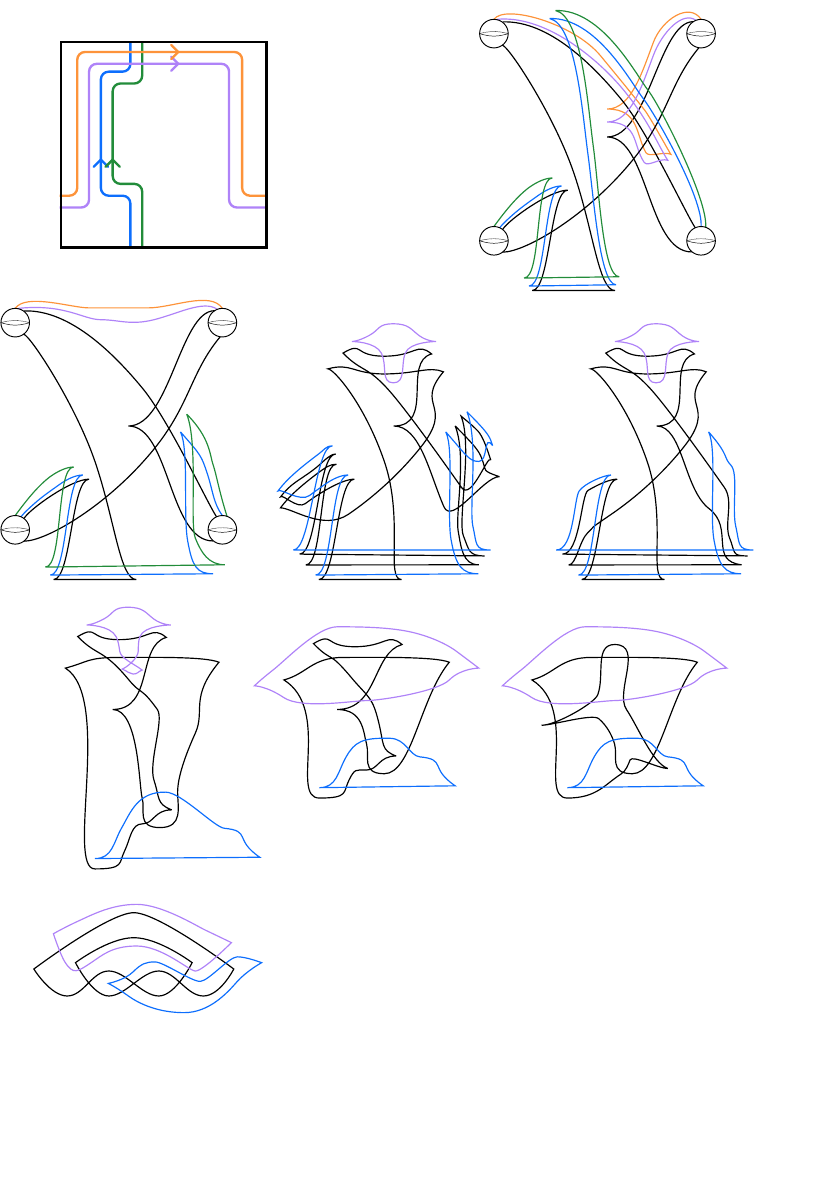
    \caption{Simplifications for PIII(D6). 
    }
    
    \label{fig:p3d6}
\end{figure}

\begin{figure}
    \centering
    \includegraphics[width=0.95\textwidth]{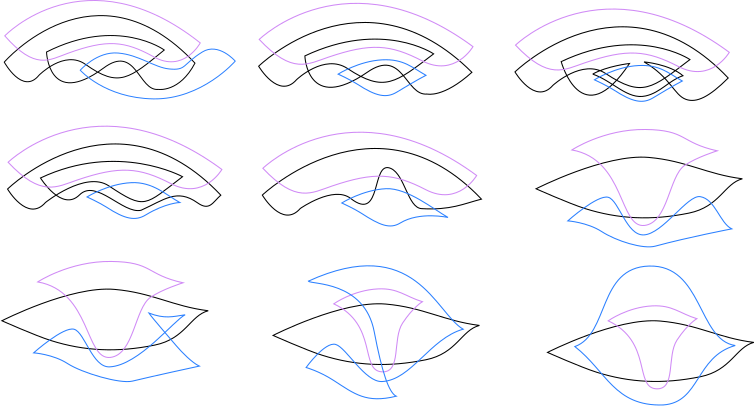}
    \caption{Simplification for \(\PIIIDsix\), continued}
    \label{fig:PIIID6-2}
\end{figure}

\begin{figure}
    \centering
    \includegraphics[width=0.95\linewidth]{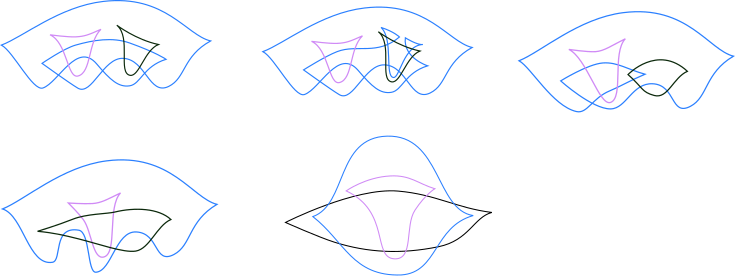}
    \caption{Simplification for \(\PIIIDsix\), from another direction.}
    \label{fig:PIIID6-3}
\end{figure}

\begin{figure}[htbp]
    \centering
    \def\svgwidth{0.95\textwidth}
    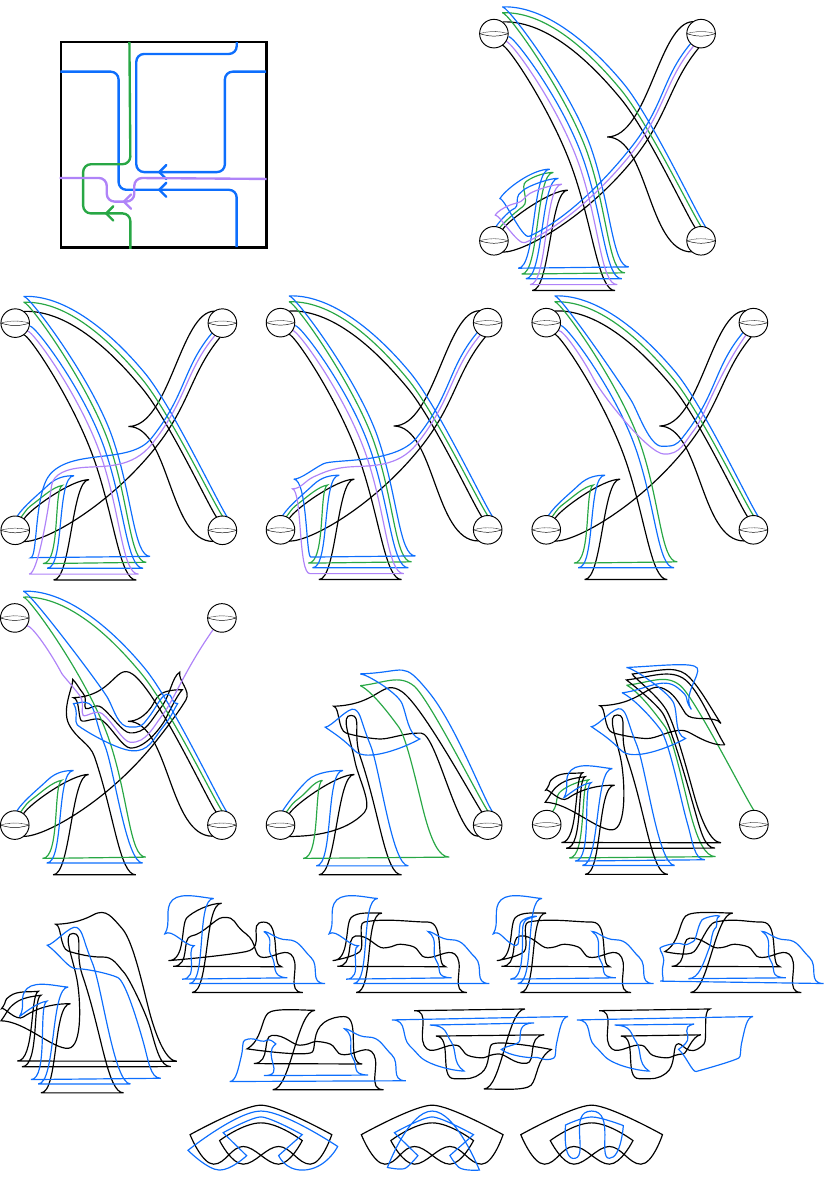
    \caption{Simplifications for PIII(D7). Observe that the second to last row contains a 180-degree rotation of the projection.}
    
    \label{fig:p3d7}
\end{figure}

\begin{figure}[htbp]
    \centering
    \def\svgwidth{0.95\textwidth}
    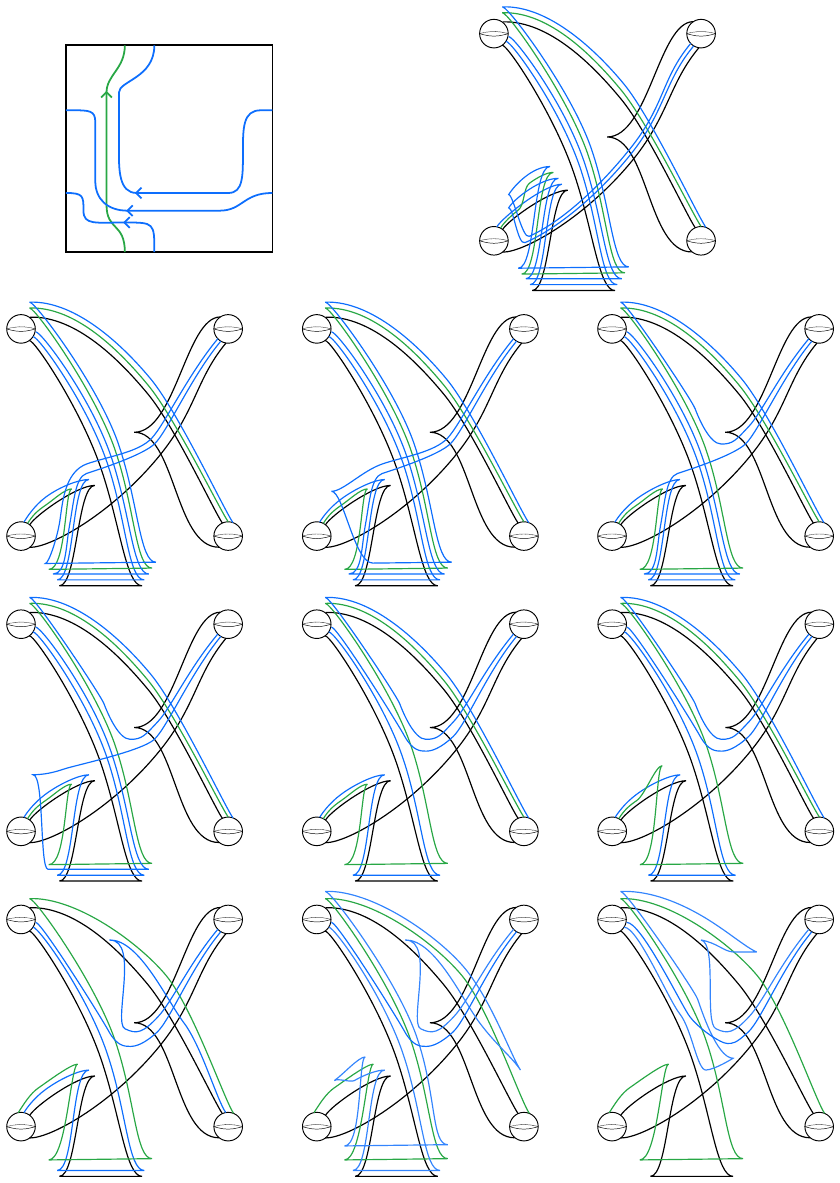
    \caption{First set of simplifications for PIII(D8).}
    
    \label{fig:p3d8_1}
\end{figure}
\begin{figure}[htbp]
    \centering
    \def\svgwidth{0.95\textwidth}
    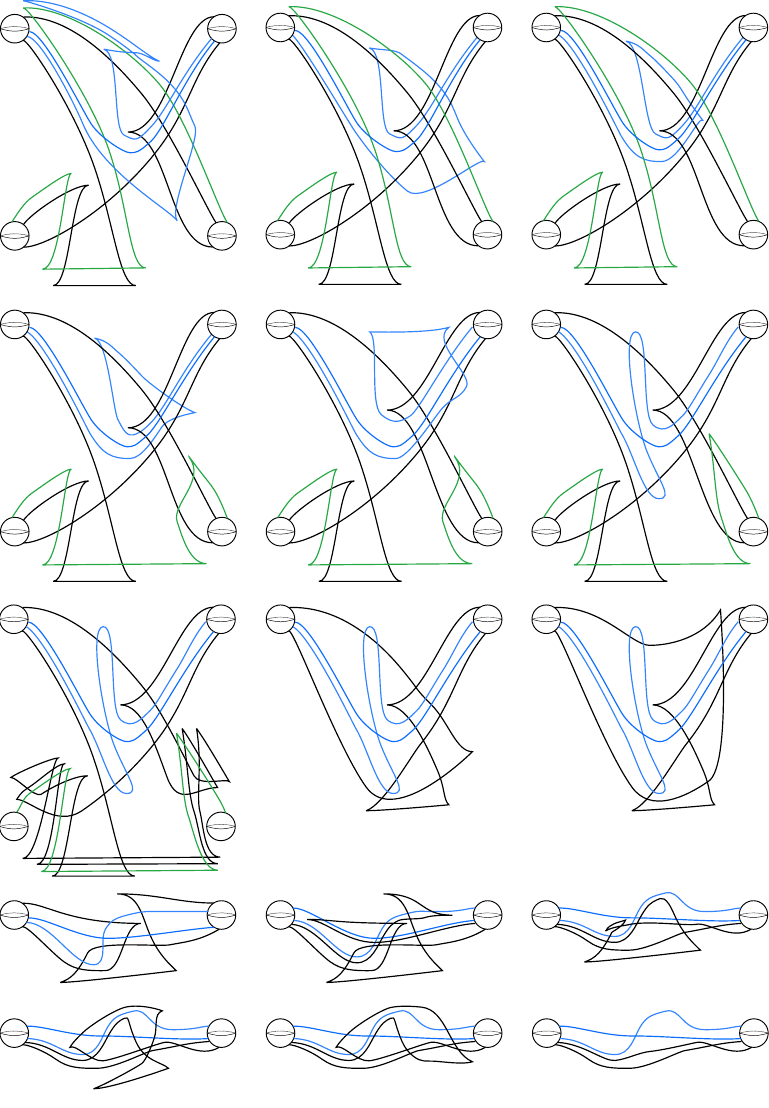
    \caption{Second set of simplifications for PIII(D8).}
    
    \label{fig:p3d8_2}
\end{figure}

\begin{figure}[htbp]
    \centering
    \def\svgwidth{0.95\textwidth}
    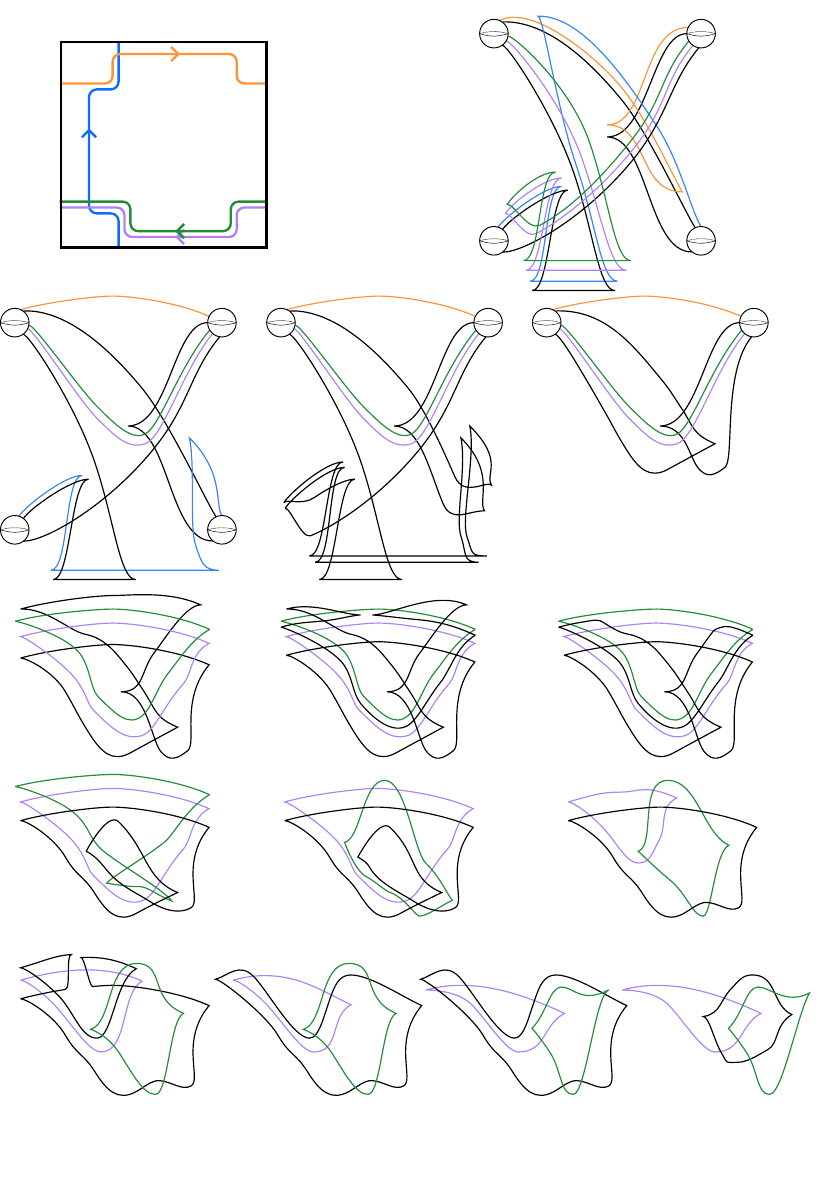
    \caption{Simplifications for PIV.}
    
    \label{fig:p4}
\end{figure}

\begin{figure}[htbp]
    \centering
    \def\svgwidth{0.95\textwidth}
    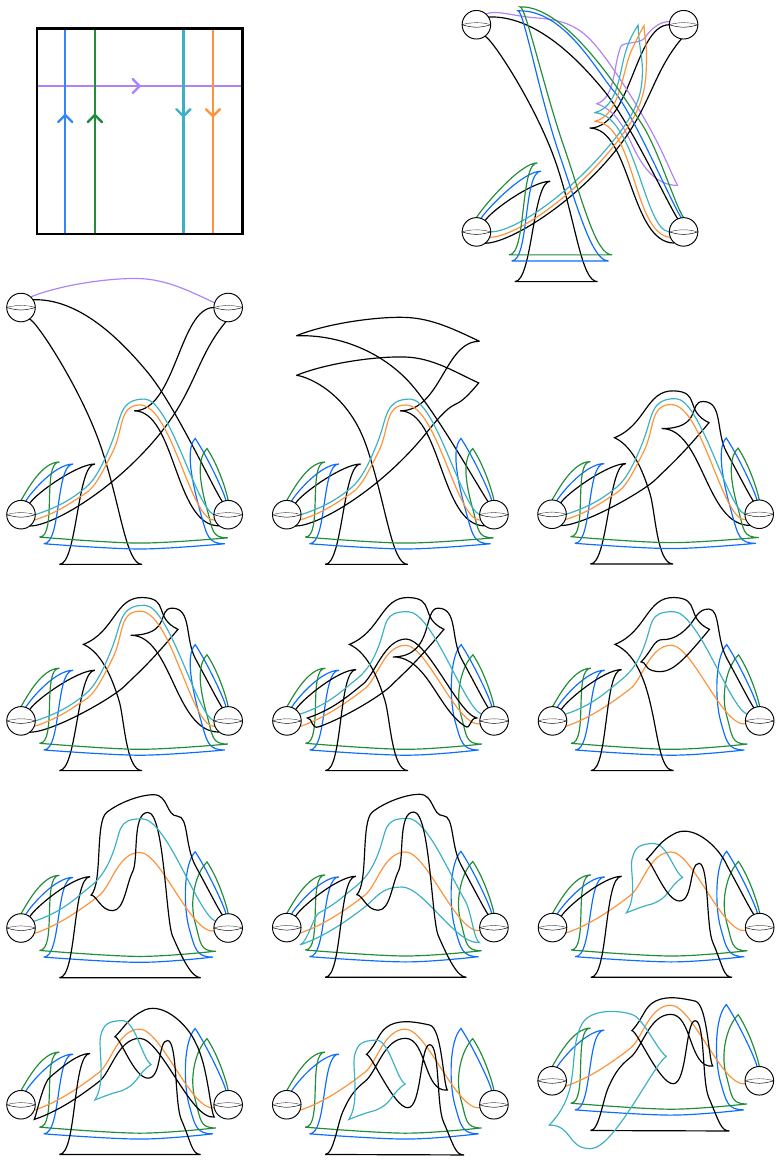
    \caption{First set of simplifications for PV.}
    
    \label{fig:p5}
\end{figure}

\begin{figure}[htbp]
    \centering
    \def\svgwidth{0.95\textwidth}
    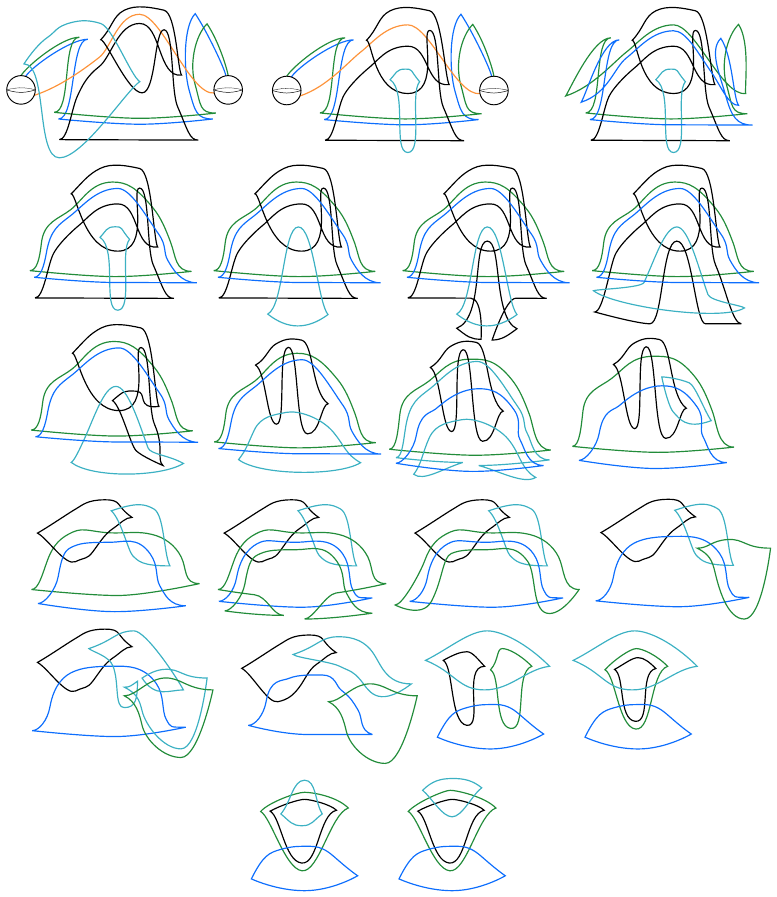
    \caption{Second set of simplifications for PV.}
    
    \label{fig:p5_2}
\end{figure}

\begin{figure}[htbp]
    \centering
    \scalebox{1}[1.03]{%
    \includegraphics[width = 1\linewidth]{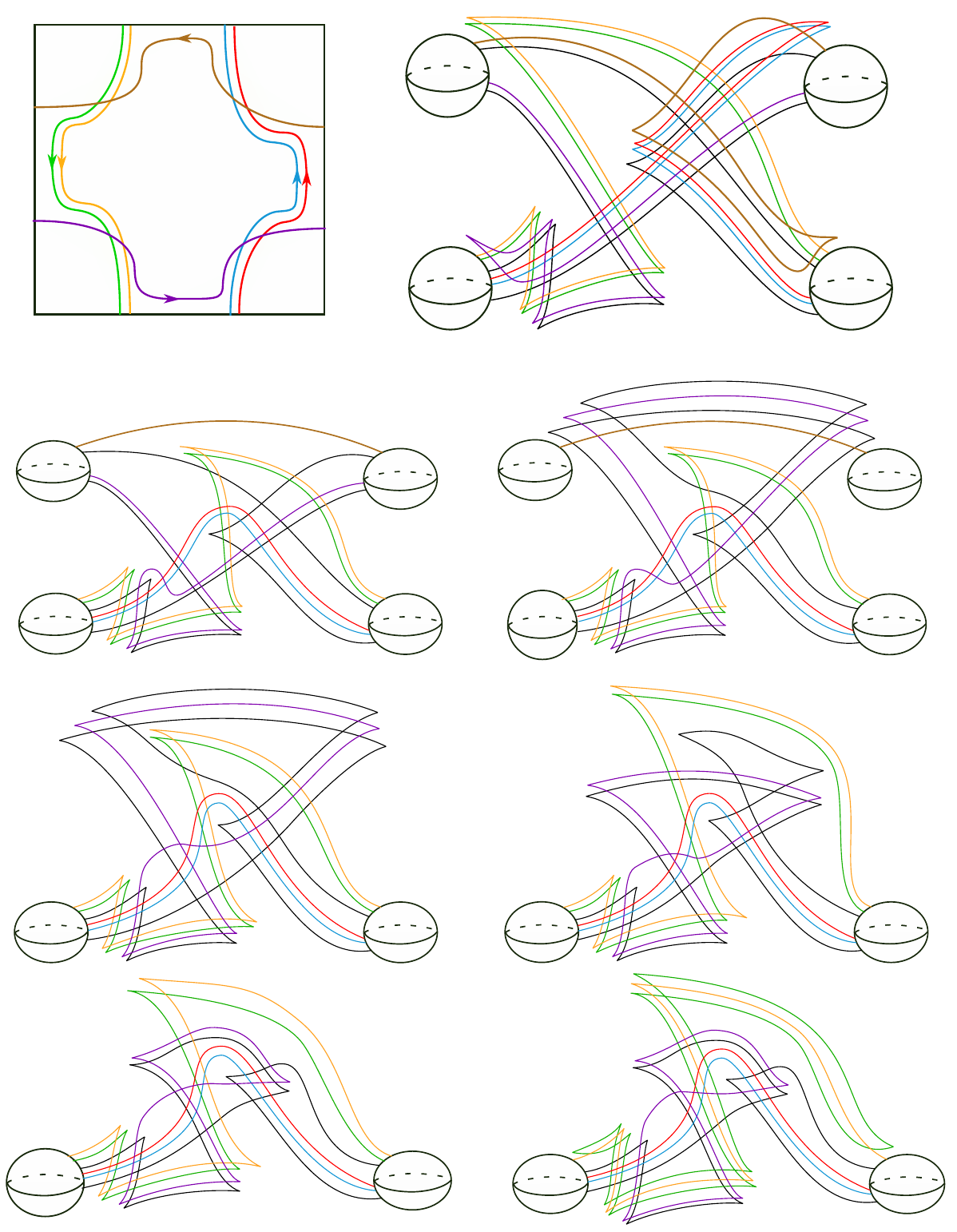}%
    }
    \caption{The first set of simplifications of \(\PVI\).}
    \label{fig:PVI1}
\end{figure}

\begin{figure}[htbp]
    \centering
        \scalebox{1}[1.03]{%
        \includegraphics[width = 1\linewidth]{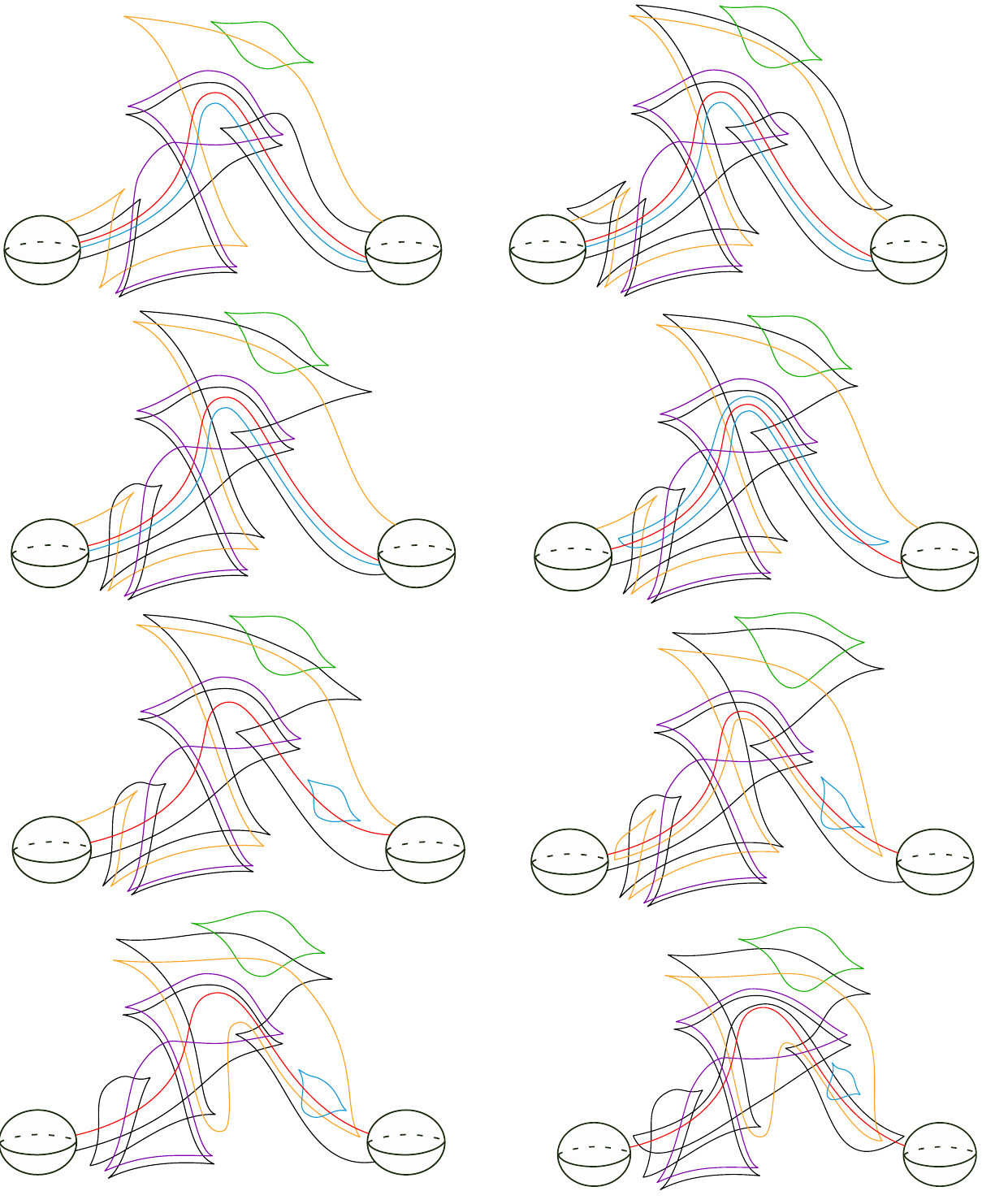}%
        } 
        \caption{The second set of simplifications of \(\PVI\).}
        \label{fig:PVI2}
    \end{figure}

    \begin{figure}[htbp]
        \centering
        \includegraphics[width = 1\linewidth]{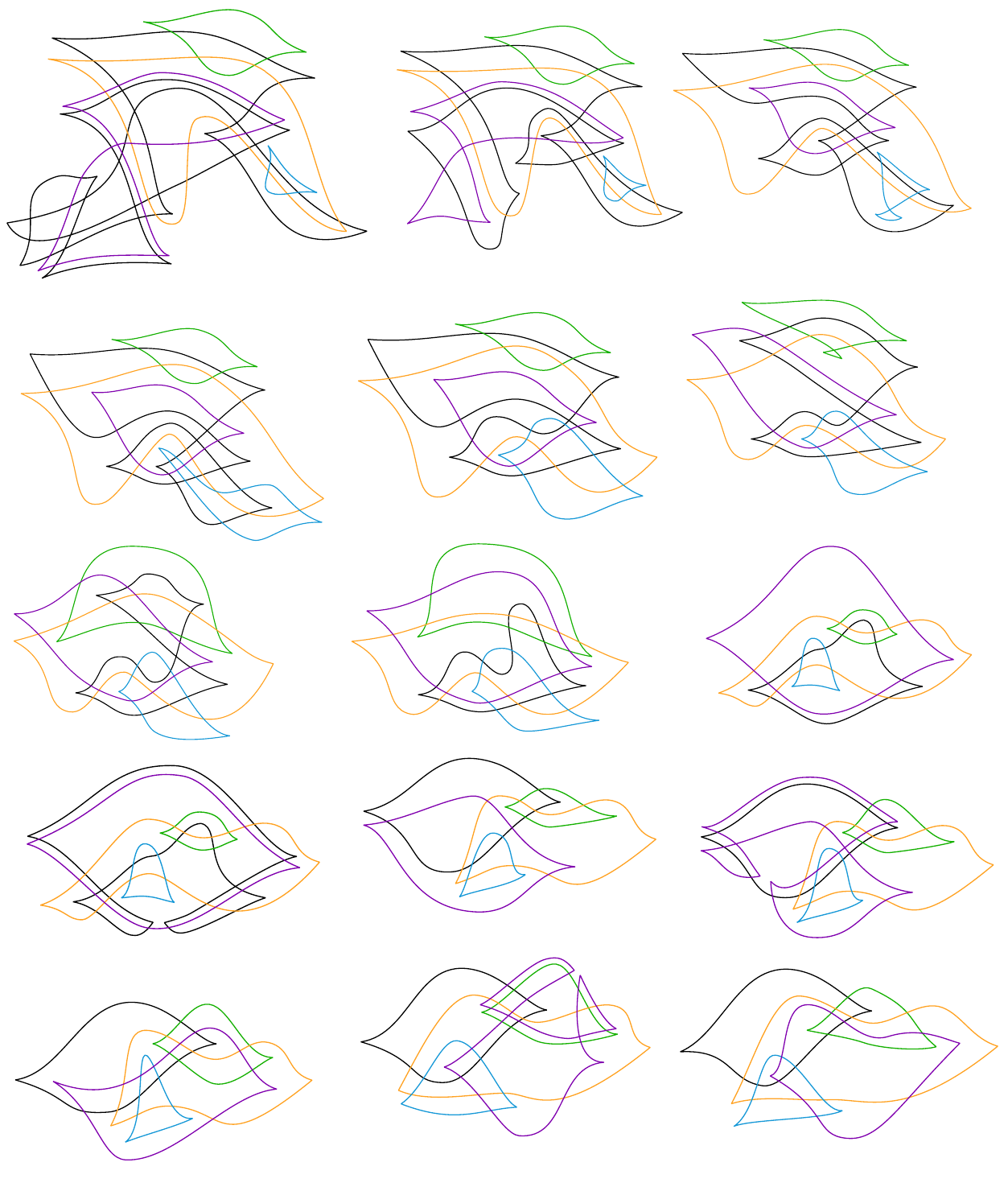}
        \caption{The third set of simplifications of \(\PVI\).}
        \label{fig:PVI3}
    \end{figure}

    \begin{figure}[htbp]
        \centering
        \includegraphics[width=1\linewidth]{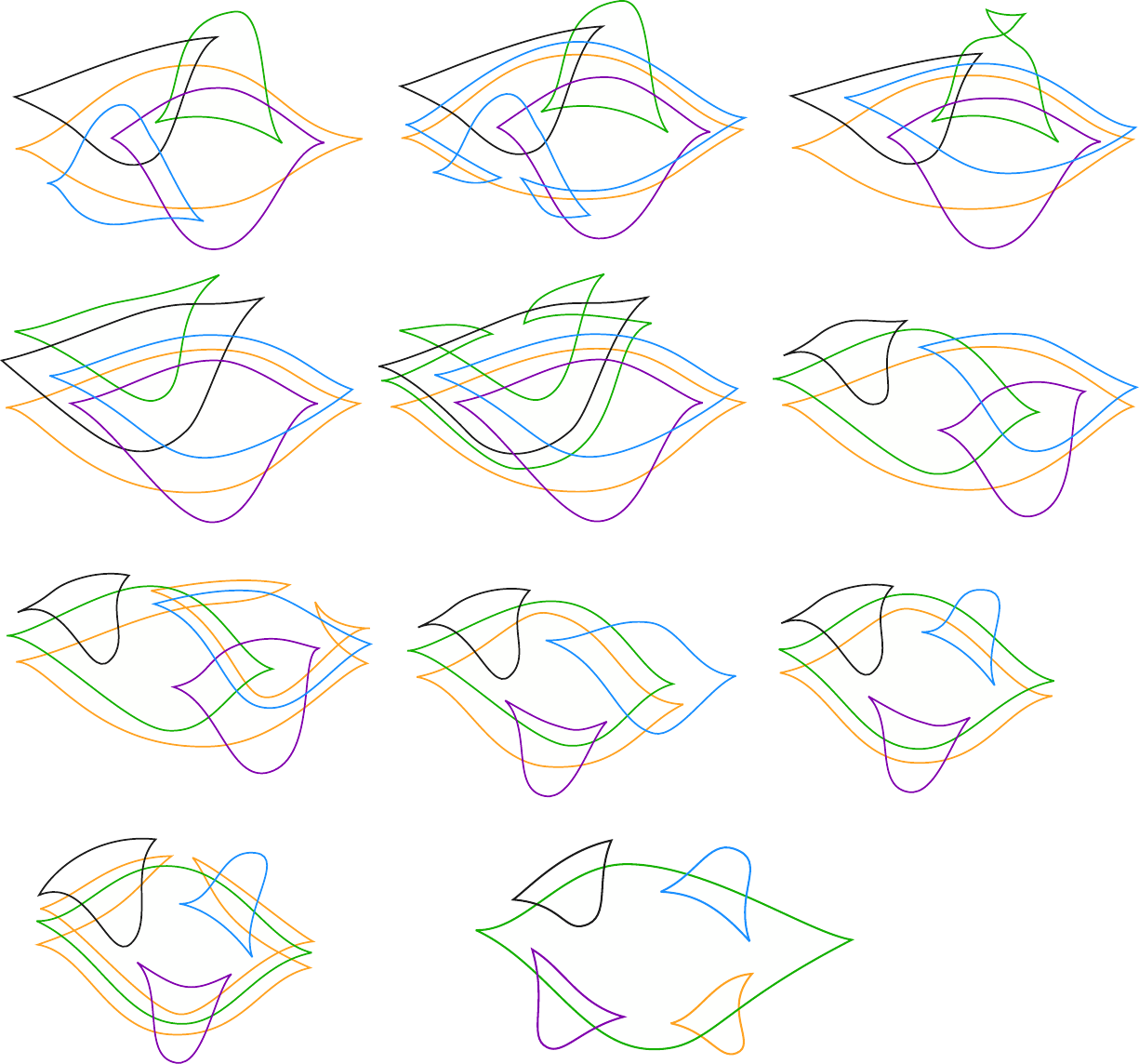}
        \caption{The last set of simplifications of \(\PVI\).}
        \label{fig:PVI-continued}
    \end{figure}

%% file: images/reidemeister.pdf_tex
\begingroup%
  \makeatletter%
  \providecommand\color[2][]{%
    \errmessage{(Inkscape) Color is used for the text in Inkscape, but the package 'color.sty' is not loaded}%
    \renewcommand\color[2][]{}%
  }%
  \providecommand\transparent[1]{%
    \errmessage{(Inkscape) Transparency is used (non-zero) for the text in Inkscape, but the package 'transparent.sty' is not loaded}%
    \renewcommand\transparent[1]{}%
  }%
  \providecommand\rotatebox[2]{#2}%
  \newcommand*\fsize{\dimexpr\f@size pt\relax}%
  \newcommand*\lineheight[1]{\fontsize{\fsize}{#1\fsize}\selectfont}%
  \ifx\svgwidth\undefined%
    \setlength{\unitlength}{126.36566523bp}%
    \ifx\svgscale\undefined%
      \relax%
    \else%
      \setlength{\unitlength}{\unitlength * \real{\svgscale}}%
    \fi%
  \else%
    \setlength{\unitlength}{\svgwidth}%
  \fi%
  \global\let\svgwidth\undefined%
  \global\let\svgscale\undefined%
  \makeatother%
  \begin{picture}(1,0.8473048)%
    \lineheight{1}%
    \setlength\tabcolsep{0pt}%
    \put(0,0){\includegraphics[width=\unitlength,page=1]{reidemeister.pdf}}%
  \end{picture}%
\endgroup%

%% file: images/gompf.pdf_tex
\begingroup%
  \makeatletter%
  \providecommand\color[2][]{%
    \errmessage{(Inkscape) Color is used for the text in Inkscape, but the package 'color.sty' is not loaded}%
    \renewcommand\color[2][]{}%
  }%
  \providecommand\transparent[1]{%
    \errmessage{(Inkscape) Transparency is used (non-zero) for the text in Inkscape, but the package 'transparent.sty' is not loaded}%
    \renewcommand\transparent[1]{}%
  }%
  \providecommand\rotatebox[2]{#2}%
  \newcommand*\fsize{\dimexpr\f@size pt\relax}%
  \newcommand*\lineheight[1]{\fontsize{\fsize}{#1\fsize}\selectfont}%
  \ifx\svgwidth\undefined%
    \setlength{\unitlength}{266.00882703bp}%
    \ifx\svgscale\undefined%
      \relax%
    \else%
      \setlength{\unitlength}{\unitlength * \real{\svgscale}}%
    \fi%
  \else%
    \setlength{\unitlength}{\svgwidth}%
  \fi%
  \global\let\svgwidth\undefined%
  \global\let\svgscale\undefined%
  \makeatother%
  \begin{picture}(1,0.306707)%
    \lineheight{1}%
    \setlength\tabcolsep{0pt}%
    \put(0,0){\includegraphics[width=\unitlength,page=1]{gompf.pdf}}%
  \end{picture}%
\endgroup%

%% file: images/slide_cancel.pdf_tex
\begingroup%
  \makeatletter%
  \providecommand\color[2][]{%
    \errmessage{(Inkscape) Color is used for the text in Inkscape, but the package 'color.sty' is not loaded}%
    \renewcommand\color[2][]{}%
  }%
  \providecommand\transparent[1]{%
    \errmessage{(Inkscape) Transparency is used (non-zero) for the text in Inkscape, but the package 'transparent.sty' is not loaded}%
    \renewcommand\transparent[1]{}%
  }%
  \providecommand\rotatebox[2]{#2}%
  \newcommand*\fsize{\dimexpr\f@size pt\relax}%
  \newcommand*\lineheight[1]{\fontsize{\fsize}{#1\fsize}\selectfont}%
  \ifx\svgwidth\undefined%
    \setlength{\unitlength}{203.29512468bp}%
    \ifx\svgscale\undefined%
      \relax%
    \else%
      \setlength{\unitlength}{\unitlength * \real{\svgscale}}%
    \fi%
  \else%
    \setlength{\unitlength}{\svgwidth}%
  \fi%
  \global\let\svgwidth\undefined%
  \global\let\svgscale\undefined%
  \makeatother%
  \begin{picture}(1,0.29963189)%
    \lineheight{1}%
    \setlength\tabcolsep{0pt}%
    \put(0,0){\includegraphics[width=\unitlength,page=1]{slide_cancel.pdf}}%
    \put(0.76060406,0.05581727){\color[rgb]{0,0,0}\makebox(0,0)[lt]{\lineheight{1.25}\smash{\begin{tabular}[t]{l}$T'$\end{tabular}}}}%
    \put(0.76192173,0.2303333){\color[rgb]{0,0,0}\makebox(0,0)[lt]{\lineheight{1.25}\smash{\begin{tabular}[t]{l}$T$\end{tabular}}}}%
  \end{picture}%
\endgroup%

%% file: images/r1_cusp.pdf_tex
\begingroup%
  \makeatletter%
  \providecommand\color[2][]{%
    \errmessage{(Inkscape) Color is used for the text in Inkscape, but the package 'color.sty' is not loaded}%
    \renewcommand\color[2][]{}%
  }%
  \providecommand\transparent[1]{%
    \errmessage{(Inkscape) Transparency is used (non-zero) for the text in Inkscape, but the package 'transparent.sty' is not loaded}%
    \renewcommand\transparent[1]{}%
  }%
  \providecommand\rotatebox[2]{#2}%
  \newcommand*\fsize{\dimexpr\f@size pt\relax}%
  \newcommand*\lineheight[1]{\fontsize{\fsize}{#1\fsize}\selectfont}%
  \ifx\svgwidth\undefined%
    \setlength{\unitlength}{127.55905512bp}%
    \ifx\svgscale\undefined%
      \relax%
    \else%
      \setlength{\unitlength}{\unitlength * \real{\svgscale}}%
    \fi%
  \else%
    \setlength{\unitlength}{\svgwidth}%
  \fi%
  \global\let\svgwidth\undefined%
  \global\let\svgscale\undefined%
  \makeatother%
  \begin{picture}(1,0.86302219)%
    \lineheight{1}%
    \setlength\tabcolsep{0pt}%
    \put(0,0){\includegraphics[width=\unitlength,page=1]{r1_cusp.pdf}}%
    \put(0.03856884,0.69669519){\color[rgb]{0,0,0}\rotatebox{-90}{\makebox(0,0)[lt]{\lineheight{1.25}\smash{\begin{tabular}[t]{l}...\end{tabular}}}}}%
    \put(0,0){\includegraphics[width=\unitlength,page=2]{r1_cusp.pdf}}%
    \put(0.54967999,0.69669519){\color[rgb]{0,0,0}\rotatebox{-90}{\makebox(0,0)[lt]{\lineheight{1.25}\smash{\begin{tabular}[t]{l}...\end{tabular}}}}}%
    \put(0,0){\includegraphics[width=\unitlength,page=3]{r1_cusp.pdf}}%
    \put(0.03856884,0.27447297){\color[rgb]{0,0,0}\rotatebox{-90}{\makebox(0,0)[lt]{\lineheight{1.25}\smash{\begin{tabular}[t]{l}...\end{tabular}}}}}%
    \put(0,0){\includegraphics[width=\unitlength,page=4]{r1_cusp.pdf}}%
    \put(0.54967999,0.27447297){\color[rgb]{0,0,0}\rotatebox{-90}{\makebox(0,0)[lt]{\lineheight{1.25}\smash{\begin{tabular}[t]{l}...\end{tabular}}}}}%
  \end{picture}%
\endgroup%

%% file: images/p6_mutation.pdf_tex
\begingroup%
  \makeatletter%
  \providecommand\color[2][]{%
    \errmessage{(Inkscape) Color is used for the text in Inkscape, but the package 'color.sty' is not loaded}%
    \renewcommand\color[2][]{}%
  }%
  \providecommand\transparent[1]{%
    \errmessage{(Inkscape) Transparency is used (non-zero) for the text in Inkscape, but the package 'transparent.sty' is not loaded}%
    \renewcommand\transparent[1]{}%
  }%
  \providecommand\rotatebox[2]{#2}%
  \newcommand*\fsize{\dimexpr\f@size pt\relax}%
  \newcommand*\lineheight[1]{\fontsize{\fsize}{#1\fsize}\selectfont}%
  \ifx\svgwidth\undefined%
    \setlength{\unitlength}{223.19149468bp}%
    \ifx\svgscale\undefined%
      \relax%
    \else%
      \setlength{\unitlength}{\unitlength * \real{\svgscale}}%
    \fi%
  \else%
    \setlength{\unitlength}{\svgwidth}%
  \fi%
  \global\let\svgwidth\undefined%
  \global\let\svgscale\undefined%
  \makeatother%
  \begin{picture}(1,0.22324954)%
    \lineheight{1}%
    \setlength\tabcolsep{0pt}%
    \put(0,0){\includegraphics[width=\unitlength,page=1]{p6_mutation.pdf}}%
  \end{picture}%
\endgroup%

%% file: images/p1.pdf_tex
\begingroup%
  \makeatletter%
  \providecommand\color[2][]{%
    \errmessage{(Inkscape) Color is used for the text in Inkscape, but the package 'color.sty' is not loaded}%
    \renewcommand\color[2][]{}%
  }%
  \providecommand\transparent[1]{%
    \errmessage{(Inkscape) Transparency is used (non-zero) for the text in Inkscape, but the package 'transparent.sty' is not loaded}%
    \renewcommand\transparent[1]{}%
  }%
  \providecommand\rotatebox[2]{#2}%
  \newcommand*\fsize{\dimexpr\f@size pt\relax}%
  \newcommand*\lineheight[1]{\fontsize{\fsize}{#1\fsize}\selectfont}%
  \ifx\svgwidth\undefined%
    \setlength{\unitlength}{388.50805087bp}%
    \ifx\svgscale\undefined%
      \relax%
    \else%
      \setlength{\unitlength}{\unitlength * \real{\svgscale}}%
    \fi%
  \else%
    \setlength{\unitlength}{\svgwidth}%
  \fi%
  \global\let\svgwidth\undefined%
  \global\let\svgscale\undefined%
  \makeatother%
  \begin{picture}(1,0.94948814)%
    \lineheight{1}%
    \setlength\tabcolsep{0pt}%
    \put(0,0){\includegraphics[width=\unitlength,page=1]{p1.pdf}}%
  \end{picture}%
\endgroup%

%% file: images/p2.pdf_tex
\begingroup%
  \makeatletter%
  \providecommand\color[2][]{%
    \errmessage{(Inkscape) Color is used for the text in Inkscape, but the package 'color.sty' is not loaded}%
    \renewcommand\color[2][]{}%
  }%
  \providecommand\transparent[1]{%
    \errmessage{(Inkscape) Transparency is used (non-zero) for the text in Inkscape, but the package 'transparent.sty' is not loaded}%
    \renewcommand\transparent[1]{}%
  }%
  \providecommand\rotatebox[2]{#2}%
  \newcommand*\fsize{\dimexpr\f@size pt\relax}%
  \newcommand*\lineheight[1]{\fontsize{\fsize}{#1\fsize}\selectfont}%
  \ifx\svgwidth\undefined%
    \setlength{\unitlength}{387.90709007bp}%
    \ifx\svgscale\undefined%
      \relax%
    \else%
      \setlength{\unitlength}{\unitlength * \real{\svgscale}}%
    \fi%
  \else%
    \setlength{\unitlength}{\svgwidth}%
  \fi%
  \global\let\svgwidth\undefined%
  \global\let\svgscale\undefined%
  \makeatother%
  \begin{picture}(1,1.44783511)%
    \lineheight{1}%
    \setlength\tabcolsep{0pt}%
    \put(0,0){\includegraphics[width=\unitlength,page=1]{p2.pdf}}%
  \end{picture}%
\endgroup%

%% file: images/p3d6.pdf_tex
\begingroup%
  \makeatletter%
  \providecommand\color[2][]{%
    \errmessage{(Inkscape) Color is used for the text in Inkscape, but the package 'color.sty' is not loaded}%
    \renewcommand\color[2][]{}%
  }%
  \providecommand\transparent[1]{%
    \errmessage{(Inkscape) Transparency is used (non-zero) for the text in Inkscape, but the package 'transparent.sty' is not loaded}%
    \renewcommand\transparent[1]{}%
  }%
  \providecommand\rotatebox[2]{#2}%
  \newcommand*\fsize{\dimexpr\f@size pt\relax}%
  \newcommand*\lineheight[1]{\fontsize{\fsize}{#1\fsize}\selectfont}%
  \ifx\svgwidth\undefined%
    \setlength{\unitlength}{403.05827464bp}%
    \ifx\svgscale\undefined%
      \relax%
    \else%
      \setlength{\unitlength}{\unitlength * \real{\svgscale}}%
    \fi%
  \else%
    \setlength{\unitlength}{\svgwidth}%
  \fi%
  \global\let\svgwidth\undefined%
  \global\let\svgscale\undefined%
  \makeatother%
  \begin{picture}(1,1.40824949)%
    \lineheight{1}%
    \setlength\tabcolsep{0pt}%
    \put(0,0){\includegraphics[width=\unitlength,page=1]{p3d6.pdf}}%
  \end{picture}%
\endgroup%

%% file: images/p3d7.pdf_tex
\begingroup%
  \makeatletter%
  \providecommand\color[2][]{%
    \errmessage{(Inkscape) Color is used for the text in Inkscape, but the package 'color.sty' is not loaded}%
    \renewcommand\color[2][]{}%
  }%
  \providecommand\transparent[1]{%
    \errmessage{(Inkscape) Transparency is used (non-zero) for the text in Inkscape, but the package 'transparent.sty' is not loaded}%
    \renewcommand\transparent[1]{}%
  }%
  \providecommand\rotatebox[2]{#2}%
  \newcommand*\fsize{\dimexpr\f@size pt\relax}%
  \newcommand*\lineheight[1]{\fontsize{\fsize}{#1\fsize}\selectfont}%
  \ifx\svgwidth\undefined%
    \setlength{\unitlength}{403.05827464bp}%
    \ifx\svgscale\undefined%
      \relax%
    \else%
      \setlength{\unitlength}{\unitlength * \real{\svgscale}}%
    \fi%
  \else%
    \setlength{\unitlength}{\svgwidth}%
  \fi%
  \global\let\svgwidth\undefined%
  \global\let\svgscale\undefined%
  \makeatother%
  \begin{picture}(1,1.40824949)%
    \lineheight{1}%
    \setlength\tabcolsep{0pt}%
    \put(0,0){\includegraphics[width=\unitlength,page=1]{p3d7.pdf}}%
  \end{picture}%
\endgroup%

%% file: images/p3d8.pdf_tex
\begingroup%
  \makeatletter%
  \providecommand\color[2][]{%
    \errmessage{(Inkscape) Color is used for the text in Inkscape, but the package 'color.sty' is not loaded}%
    \renewcommand\color[2][]{}%
  }%
  \providecommand\transparent[1]{%
    \errmessage{(Inkscape) Transparency is used (non-zero) for the text in Inkscape, but the package 'transparent.sty' is not loaded}%
    \renewcommand\transparent[1]{}%
  }%
  \providecommand\rotatebox[2]{#2}%
  \newcommand*\fsize{\dimexpr\f@size pt\relax}%
  \newcommand*\lineheight[1]{\fontsize{\fsize}{#1\fsize}\selectfont}%
  \ifx\svgwidth\undefined%
    \setlength{\unitlength}{403.05827464bp}%
    \ifx\svgscale\undefined%
      \relax%
    \else%
      \setlength{\unitlength}{\unitlength * \real{\svgscale}}%
    \fi%
  \else%
    \setlength{\unitlength}{\svgwidth}%
  \fi%
  \global\let\svgwidth\undefined%
  \global\let\svgscale\undefined%
  \makeatother%
  \begin{picture}(1,1.40824949)%
    \lineheight{1}%
    \setlength\tabcolsep{0pt}%
    \put(0,0){\includegraphics[width=\unitlength,page=1]{p3d8.pdf}}%
  \end{picture}%
\endgroup%

%% file: images/p3d8_2.pdf_tex
\begingroup%
  \makeatletter%
  \providecommand\color[2][]{%
    \errmessage{(Inkscape) Color is used for the text in Inkscape, but the package 'color.sty' is not loaded}%
    \renewcommand\color[2][]{}%
  }%
  \providecommand\transparent[1]{%
    \errmessage{(Inkscape) Transparency is used (non-zero) for the text in Inkscape, but the package 'transparent.sty' is not loaded}%
    \renewcommand\transparent[1]{}%
  }%
  \providecommand\rotatebox[2]{#2}%
  \newcommand*\fsize{\dimexpr\f@size pt\relax}%
  \newcommand*\lineheight[1]{\fontsize{\fsize}{#1\fsize}\selectfont}%
  \ifx\svgwidth\undefined%
    \setlength{\unitlength}{368.96782822bp}%
    \ifx\svgscale\undefined%
      \relax%
    \else%
      \setlength{\unitlength}{\unitlength * \real{\svgscale}}%
    \fi%
  \else%
    \setlength{\unitlength}{\svgwidth}%
  \fi%
  \global\let\svgwidth\undefined%
  \global\let\svgscale\undefined%
  \makeatother%
  \begin{picture}(1,1.44019622)%
    \lineheight{1}%
    \setlength\tabcolsep{0pt}%
    \put(0,0){\includegraphics[width=\unitlength,page=1]{p3d8_2.pdf}}%
  \end{picture}%
\endgroup%

%% file: images/p4.pdf_tex
\begingroup%
  \makeatletter%
  \providecommand\color[2][]{%
    \errmessage{(Inkscape) Color is used for the text in Inkscape, but the package 'color.sty' is not loaded}%
    \renewcommand\color[2][]{}%
  }%
  \providecommand\transparent[1]{%
    \errmessage{(Inkscape) Transparency is used (non-zero) for the text in Inkscape, but the package 'transparent.sty' is not loaded}%
    \renewcommand\transparent[1]{}%
  }%
  \providecommand\rotatebox[2]{#2}%
  \newcommand*\fsize{\dimexpr\f@size pt\relax}%
  \newcommand*\lineheight[1]{\fontsize{\fsize}{#1\fsize}\selectfont}%
  \ifx\svgwidth\undefined%
    \setlength{\unitlength}{403.05827464bp}%
    \ifx\svgscale\undefined%
      \relax%
    \else%
      \setlength{\unitlength}{\unitlength * \real{\svgscale}}%
    \fi%
  \else%
    \setlength{\unitlength}{\svgwidth}%
  \fi%
  \global\let\svgwidth\undefined%
  \global\let\svgscale\undefined%
  \makeatother%
  \begin{picture}(1,1.40824949)%
    \lineheight{1}%
    \setlength\tabcolsep{0pt}%
    \put(0,0){\includegraphics[width=\unitlength,page=1]{p4.pdf}}%
  \end{picture}%
\endgroup%

%% file: images/p5.pdf_tex
\begingroup%
  \makeatletter%
  \providecommand\color[2][]{%
    \errmessage{(Inkscape) Color is used for the text in Inkscape, but the package 'color.sty' is not loaded}%
    \renewcommand\color[2][]{}%
  }%
  \providecommand\transparent[1]{%
    \errmessage{(Inkscape) Transparency is used (non-zero) for the text in Inkscape, but the package 'transparent.sty' is not loaded}%
    \renewcommand\transparent[1]{}%
  }%
  \providecommand\rotatebox[2]{#2}%
  \newcommand*\fsize{\dimexpr\f@size pt\relax}%
  \newcommand*\lineheight[1]{\fontsize{\fsize}{#1\fsize}\selectfont}%
  \ifx\svgwidth\undefined%
    \setlength{\unitlength}{374.40000519bp}%
    \ifx\svgscale\undefined%
      \relax%
    \else%
      \setlength{\unitlength}{\unitlength * \real{\svgscale}}%
    \fi%
  \else%
    \setlength{\unitlength}{\svgwidth}%
  \fi%
  \global\let\svgwidth\undefined%
  \global\let\svgscale\undefined%
  \makeatother%
  \begin{picture}(1,1.48860541)%
    \lineheight{1}%
    \setlength\tabcolsep{0pt}%
    \put(0,0){\includegraphics[width=\unitlength,page=1]{p5.pdf}}%
  \end{picture}%
\endgroup%

%% file: images/p5_2.pdf_tex
\begingroup%
  \makeatletter%
  \providecommand\color[2][]{%
    \errmessage{(Inkscape) Color is used for the text in Inkscape, but the package 'color.sty' is not loaded}%
    \renewcommand\color[2][]{}%
  }%
  \providecommand\transparent[1]{%
    \errmessage{(Inkscape) Transparency is used (non-zero) for the text in Inkscape, but the package 'transparent.sty' is not loaded}%
    \renewcommand\transparent[1]{}%
  }%
  \providecommand\rotatebox[2]{#2}%
  \newcommand*\fsize{\dimexpr\f@size pt\relax}%
  \newcommand*\lineheight[1]{\fontsize{\fsize}{#1\fsize}\selectfont}%
  \ifx\svgwidth\undefined%
    \setlength{\unitlength}{372.85793995bp}%
    \ifx\svgscale\undefined%
      \relax%
    \else%
      \setlength{\unitlength}{\unitlength * \real{\svgscale}}%
    \fi%
  \else%
    \setlength{\unitlength}{\svgwidth}%
  \fi%
  \global\let\svgwidth\undefined%
  \global\let\svgscale\undefined%
  \makeatother%
  \begin{picture}(1,1.15456612)%
    \lineheight{1}%
    \setlength\tabcolsep{0pt}%
    \put(0,0){\includegraphics[width=\unitlength,page=1]{p5_2.pdf}}%
  \end{picture}%
\endgroup%

%% file: sections/mirrorsymmetry.tex

\section{Recollections on mirror symmetry for toric varieties.} \label{sec:HMS}

We review some definitions and results, especially from \cite{FLTZ, Kuwagaki-CCC,  gammage-le}. 

\begin{definition} \cite{Geraschenko-Satriano-toric}
    A stacky fan is a collection \(\mathbf{\Sigma} = (L, N, \widehat{\Sigma}, \beta)\), where \(\widehat{\Sigma}\) is a fan in a lattice \(L\), and \(\beta: L \rightarrow N\) is a homomorphism between lattices with finite cokernel. 
Let \(X_{\widehat{\Sigma}}\) be the toric variety associated to \(\widehat{\Sigma}\). Denote the dual lattices
by \( K = \Hom (L, \bZ)\) and \(M = \Hom (N, \bZ)\). 
The dual map \(\beta^*: M \rightarrow K \) induces a map between the algebraic tori: \(T_L \rightarrow T_N\). Take \(G_\beta\) to be the kernel of this map. The toric stack associated to \(\mathbf{\Sigma}\) is defined to be \(\cX_{\mathbf{\Sigma}} = [X_{\widehat{\Sigma}}/ G_\beta]\). 
\end{definition}

\begin{definition}{\cite{FLTZ, FLTZ-stack}}
\label{def:fltz-skeleton} 
    Let \(\mathbf{\Sigma} = (L, N, \widehat{\Sigma}, \beta)\) be a stacky fan such that \(\beta\) maps \(\widehat{\Sigma}\) to a fan \({\Sigma}\) in \(N\), and each cone \(\widehat{\sigma} \in \widehat{\Sigma}\) is mapped homeomorphically to a cone \(\sigma \in \Sigma\). 

    Take \(M_\sigma = (\beta^*)^{-1} \left( K / (\widehat{\sigma}^\perp \cap K) \right)\), which is a discrete subset of \(M_\bR / \sigma^\perp\). Suppose that \(N\) is of rank \(n\).  Then 
\begin{equation}\label{eq:fltz-skeleton}
\Lambda_{\mathbf{\Sigma}} := \bigcup_{\sigma\in \Sigma} \bigcup_{\chi \in  M_\sigma} [\sigma^\perp + \chi] \times -\sigma \subset M_{\bR} / M \times N_\bR \cong T^*T^n. 
\end{equation}    
\end{definition}

The subset $\Lambda_{\mathbf{\Sigma}}$ (often referred to as the FLTZ skeleton) 
was introduced in \cite{FLTZ} in order to give a (then conjectural, now established) microlocal characterization of a category introduced previously \cite{bondal}, and to correspondingly formulate the (then conjectural, now established) mirror symmetry:

\begin{theorem}\label{thm:ccc}\cite{Kuwagaki-CCC}
There is an equivalence of categories $
        \Coh \left( \mathcal{X}_{\mathbf{\Sigma}} \right) \simeq        \Sh (T^n, \Lambda_{\mathbf{\Sigma}})$.
    \end{theorem}

\begin{remark}
    For our purposes here, we don't require the full strength of \cite{Kuwagaki-CCC}; e.g. the case of toric surfaces established earlier in \cite{kuwagaki-surfaces} would do. 
\end{remark}

We now consider attaching handles, following the ideas of \cite{gammage-le}.  The main tool on the B-side is the following result. 

\begin{theorem}{\cite[Theorem 4.13]{gammage-le}}
\label{thm:coh-of-blow-up}
    Let \(\cX\) be a smooth variety or Deligne--Mumford stack, \(\cD \stackrel{j}{\hookrightarrow} \cX\) a divisor, and \(\cH \stackrel{i}{\hookrightarrow} \cD \) a codimension \(2\) smooth substack of \(\cX\). Denote by \(\mathrm{Bl}_{\cH} \cX\) the blow up of \(\cX\) along \(\cH\), \(\widetilde{\cD}\) the strict transform of \(\cD\), and \(\cU = \mathrm{Bl}_{\cH} \cX \setminus \widetilde{\cD}\). Then there is a pushout diagram of dg-categories:
    \begin{equation}\label{eq:pushout-on-B-side}
    \begin{tikzcd}
        \Coh(\cD) \ar[r, " i^*"] \ar[d, "j_*"'] & \Coh(\cH) \ar[d]\\
        \Coh (\cX) \ar[r] &  \Coh(\cU)
        \arrow[from=1-1, to=2-2, phantom, "\ulcorner", very near end]
    \end{tikzcd}
    \end{equation}
\end{theorem}


    

Fix now a collection of vectors \( \{k_i v_i \} \), where \(v_i\) is a primitive vector in \(N=\bZ^2\) and \(k_i\) is a positive integer.  We define a stacky fan
$\mathbf{\Sigma}$ by taking  \(L = \bigoplus_i \bZ e_i\) and \( \beta: L \rightarrow N\) sending \(e_i\) to \(k_i v_i\), and $\widehat{\Sigma}$ to be the fan formed by the rays $e_i$.

Let \(\cX_{\mathbf{\Sigma}}\) be the corresponding stacky toric surface. Its toric boundary divisor  $\partial \cX_{\mathbf{\Sigma}} = \coprod \cD_i $ consists of disjoint components, where \(\cD_i\)  corresponds to the cone \(k_i v_i\) and carries a canonical isomorphism
$
\cD_i \cong \bC^* \times B \mu_{k_i}.
$
Here \( \mu_{k_i}\) is the cyclic group \(\bZ / k_i \bZ\). 
We write $\mathbf{-1} := \coprod \{-1\} \times B \mu_{k_i}$, and 
$$\mathbf{B}_{\mathbf{\Sigma}}:= \mathrm{Bl}_{\mathbf{-1}}\cX \setminus \coprod \widetilde \cD_i,$$
where $\widetilde{\cD}_i$ denotes the strict transform under the blowup. 

On the A-side, take \(\Lambda_{\mathbf{\Sigma}}\) to the  corresponding FLTZ skeleton in \(T^*T^2\). Unwinding the construction in \zcref{def:fltz-skeleton}, the infinity part 
$\Lambda_\mathbf{\Sigma}^\infty := \Lambda_\mathbf{\Sigma} \cap S^*T^2$
is a disjoint union of \(S^1\).  The vector \(k_i v_i\) corresponds to \(k_i\)-many parallel copies of \( v_i^\perp \times (-v_i)\) (viewed as a conical subset of \(M_\bR/M \times N_\bR \cong T^*T^2\)). We illustrate in \zcref{fig: circle at infinity}.
(Note that by \cite{GKS}, translating these \(k_i\) copies along the perpendicular direction does not change the category \(\Sh(T^2, \Lambda_{\mathbf{\Sigma}})\).) 

We write \(\mathbf{A}_\Sigma\) for the Weinstein manifold obtained by attaching \(2\)-handles to \(T^*T^2\) along \(\Lambda_{\mathbf{\Sigma}}^\infty\):
\[
\mathbf{A}_{\mathbf{\Sigma}} \coloneqq T^*T^2 \cup_{\Lambda_{\mathbf{\Sigma}}^\infty } (2\mathrm{-handles}).
\]

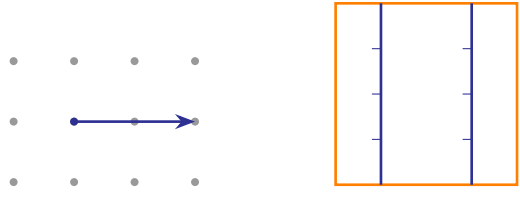
\begin{figure}
    \centering
    \begin{tikzpicture}[scale = .8]
    \foreach \x in {-1,...,2}
            \foreach \y in {-1,...,1}
             \node[circle, fill=gray!80, inner sep=0pt, minimum size=3pt] at (\x,\y) {};
    \node[circle, fill=Blue, inner sep=0pt, minimum size=3pt] at (0, 0) {};
    \draw[-Stealth, Blue, line width = 1pt] (0, 0) -- (2, 0);
\end{tikzpicture}
\qquad \qquad
\begin{tikzpicture}[scale = .6]
    \draw[orange, line width = 1pt] (0, 0) -- (4, 0) -- (4, 4) -- (0, 4) -- cycle;
    \draw[Blue, line width = 1pt] (1, 0) -- (1, 4);
    \draw[Blue, line width = 1pt] (3, 0) -- (3, 4); 
    \draw[Blue] (1, 1) -- (.8, 1); 
    \draw[Blue] (1, 3) -- (.8, 3);
    \draw[Blue] (1, 2) -- (.8, 2); 
    \draw[Blue] (3, 2) -- (2.8, 2);
    \draw[Blue] (3, 1) -- (2.8, 1);
    \draw[Blue] (3, 3) -- (2.8, 3); 
\end{tikzpicture}
    \caption{A 1d cone of a stacky fan $\mathbf \Sigma$, and the corresponding component of $\Lambda_{\mathbf \Sigma}^\infty$.}
    \label{fig: circle at infinity}
\end{figure}

\begin{theorem}{\cite[Theorem 4.12]{gammage-le}}
    \label{thm:HMS-for-U-M}
$    \Coh (\mathbf{B}_{\mathbf{\Sigma}}) \cong \mathrm{Fuk}(\mathbf{A}_{\mathbf{\Sigma}})$.
\end{theorem}

\begin{proof}
The authors of \cite{gammage-le} officially restrict attention to fans consisting of 1d cones which are linearly independent.  The proof here is essentially the same as in \cite{gammage-le}, though it is now important to take care and blow up at the particular points $\mathbf{-1}$.  We repeat the proof from \cite{gammage-le} to clarify this point. 
For visual simplicity,
we now drop the subscript \(\mathbf{\Sigma}\) from the notation.

By construction, the skeleton of $\mathbf{A}$ is  $\Lambda \cup_{\Lambda^\infty} (  \coprod D^2 )$. So 
\cite{GPS3} we have $\mathrm{Fuk}(\mathbf{A}) = \mu sh( \Lambda \cup_{\Lambda^\infty} (  \coprod D^2 ))$.  (We refer to \cite{nadler-shende} for the development of $\mu sh$.)  As $\mu sh$ is a cosheaf/sheaf of categories, we have
    \begin{equation}
    \label{eq:pushout-on-A-side}
    \begin{tikzcd}
        \mu sh (\Lambda^\infty) 
        \ar[r, " \iota_*"] \ar[d, "\mu^l"'] 
        &  \mu sh\left(\coprod D^2 \right)\ar[d] 
        \\
        \mu sh(\Lambda) \cong \Sh_\Lambda (T^2) \ar[r]
        &  \mu sh \left( \Lambda \cup_{\Lambda^\infty} (  \coprod D^2 ) \right)
        \arrow[from=1-1, to=2-2, phantom, "\ulcorner", very near end]
    \end{tikzcd}
    \end{equation}
Here, the indicated maps are the left  adjoints to the usual restriction maps of microsheaves.

    We compare the pushout diagrams~\eqref{eq:pushout-on-B-side} and~\eqref{eq:pushout-on-A-side}. As a very special case of \cite[Thm. 7.4.1]{gammage-shende}, the functor \(\mu^l\) is mirror to \(j_*\). 

    A choice of secondary Maslov data on $\Lambda^\infty$ -- which, for $\Z$-coefficients as here, is a grading and a spin structure  \cite{guillermou}\footnote{The translation to the present language can be extracted e.g. from the discussion below \cite[Thm. 4.27]{CKNS-perverse} along with \cite[Lemma A.5]{CKNS-perverse}.} -- allows us to identify
    $\mu sh (\Lambda^\infty) \cong \Loc ( \Lambda^\infty)$.
    If we wish to do this compatibly with the standard identification of $\Loc(T^2)$ with $\Coh((\C^*)^2)$, then we must choose the {\em non-bounding} spin structure on each component of $\Lambda^\infty$.  
    
    Thus, the restriction  $\iota^*$ adjoint to  \(\iota_*\) 
    \[
    \iota^*: \Loc( \coprod D^2 ) \cong \mu sh(\coprod D^2)   \longrightarrow  \mu sh (\Lambda^\infty) \cong \Loc ( \coprod S^1)
    \]
    is not the naive restriction of local systems; instead, it carries the rank one trivial local system on \(D^2\) to the rank one local system on \(S^1\) with monodromy \(-1\). 
    This completes the proof. 
\end{proof}

We give also a non-stacky mirror.  

\begin{lemma}
    $\mathbf{B}_{\mathbf{\Sigma}}$ is non-stacky away from a collection of points, one for each ray $k_i v_i$ of the fan.  The corresponding point is analytically locally $\C^2/\mu_{k_i}$, where \(\mu_{k_i}\) acts as \( \zeta \cdot (x, y) = ( \zeta x, \zeta^{-1} y) \). 
\end{lemma}
\begin{proof}
    Recall that  \(\mathbf{B}_{\mathbf{\Sigma}}:= \mathrm{Bl}_{\mathbf{-1}}\cX \setminus \coprod \widetilde \cD_i\). 
    Denote by \(\fp_i\) the stacky point on \(\cD_i\) with coordinate \(-1\). Then \(\mathbf{-1} := \coprod_i \fp_i \). Note that \(\cX \backslash \coprod \cD_i \cong (\bC^*)^2\), so we only need to look at local coordinates around each \(\fp_i\) and understand what happens when we perform the blowup.

    Take a neighborhood of \( \fp_i\), isomorphic to 
    \[ [\bC^2_{z, t} / \mu_{k_i}] \]
    where \(\mu_{k_i}\) only acts on the first coordinate:
    \(
    \zeta \cdot (z, t) = (\zeta z, t)
    \),
    such that 
    \(
    \cD_i= \{ z = 0 \}\) and \( \fp_i = (0, 0)
    \).
    After blowup there are two local charts
    \[
     \{ (z, s) \mid t = zs\}, \qquad\{ (t, w) \mid z = tw \}, 
    \]
    and \(\widetilde{\cD_i} = \{ w = 0\} \) sits in the second chart. 
    The group \(\mu_{k_i}\) acts as 
    \[
    \zeta \cdot (z, s) = (\zeta z, \zeta^{-1} s) \quad \text{and} \quad \zeta \cdot (t, w) = (t, \zeta w).  
    \]
    After removing the divisor \( \widetilde{\cD_i} \), there only remains one stacky point \( (z, s) = (0, 0)\), which has the prescribed group action.  
\end{proof}

In particular, $\mathbf{B}_{\mathbf{\Sigma}}$ is a noncommutative resolution of its coarse moduli space 
$|\mathbf{B}_{\mathbf{\Sigma}}|$, which has only $A_n$ singularities. We will use the following description of  the minimal resolution of $|\mathbf{B}_{\mathbf{\Sigma}}|$. 

\begin{construction}\label{construction:minimal_resolution}
    Let \( \cX_{\mathbf{\Sigma}}\) be a stacky surface given by a collection of vectors \( \{ k_i v_i\}_i\) as before. Its coarse moduli space \( | \cX_{\mathbf \Sigma} |\) is a toric variety given by the vectors \( \{v_i\}_i \). Denote the corresponding toric boundary divisors by \(D_i\).

    We define the smooth surface \(\cB_{\Sigma}\) to be the variety  obtained from the coarse moduli space $|\mathcal{X}_\mathbf{\Sigma}|$ 
    by, for each $i$, blowing up $-1 \in D_i$, then iteratively blowing up the point where the exceptional divisor meets the strict transform of $D_i$, for a total of $k_i$ blowups, and then finally deleting the strict transform of $D_i$.
\end{construction}

\begin{lemma}\label{lem:BSigma_resolution}
    \(\cB_{\mathbf \Sigma}\) is isomorphic to the minimal resolution of \(|\mathbf B_{ \mathbf \Sigma}|\). 
\end{lemma}
\begin{proof}
    We only need to show that they are isomorphic locally around each blowup. 
    This can be done by a local coordinate computation. Here we instead do it in a different way, using toric geometry.  
    
    First, let us consider the minimal resolution of \( |\mathbf{B}_\mathbf{\Sigma}|\). Again we choose a neighborhood of \( -1 \in {\cD_i}\), isomorphic to 
    \( [\bC/ (\mu_{k_i}) \times \bC] \). This local chart is the same as the toric stack given by the stacky cone generated by \( e_0 = (0, 1)\) and \(e_1 = (k_i, 0) \), and the toric divisor corresponding to \( e_1\) is \(\cD_i\).  
    
    Blowing up at the origin is the same as adding a vector \(e_0 + e_1\):
    
    \begin{tikzpicture}[line width = 1pt, Blue]
    \foreach \x in {0,...,3}
        \foreach \y in {0,...,1}
            \node[circle, fill=gray!80, inner sep=0pt, minimum size=3pt] at (\x,\y) {};
            
            \draw[-Stealth] (0, 0) -- (0, 1);
            \draw[-Stealth] (0, 0) -- (3, 0);
            \draw[-Stealth] (0, 0) -- (3, 1); 

            \node[right] at (3, 0) { \( e_1 \)}; 
            \node[left] at (0, 1) { \(e_0\) }; 
            \node[right] at (3, 1) { \(e_0 + e_1\) };            
    \end{tikzpicture}
    
    Then deleting \( \widetilde{\cD_i}\) is equivalent to deleting the vector \(e_1\). Taking the coarse moduli means we view the cone generated by \(e_0\) and \(e_0 + e_1\) as a normal toric fan, instead of a stacky one. Finally, taking the minimal resolution is just adding vectors \( ( 1, 1), (2, 1), \dots,  (k_i-1, 1)\):
    
    \begin{tikzpicture}[line width = 1pt, Blue]
    \foreach \x in {0,...,3}
        \foreach \y in {0,...,1}
            \node[circle, fill=gray!80, inner sep=0pt, minimum size=3pt] at (\x,\y) {};
            
            \draw[-Stealth] (0, 0) -- (0, 1);
            \draw[-Stealth] (0, 0) -- (3, 1); 
            \draw[-Stealth] (0, 0) -- (1, 1); 
            \draw[-Stealth] (0, 0) -- (2, 1);

            \node[left] at (0, 1) { \(e_0\) }; 
            \node[right] at (3, 1) { \(e_0 + e_1\) };            
    \end{tikzpicture}.

    Now let us consider \(\mathcal{B}_\mathbf{\Sigma}\). Take a neighborhood of \(-1 \in D_i\), which is isomorphic to \(\bC^2\). This local chart is isomorphic to the toric variety given by the cone generated by \(e_0 = (0, 1)\) and \(e'_1 = (1, 0)\), and the toric divisor corresponding to \(e'_1\) is \(D_i\). 
    
    \begin{tikzpicture}[line width = 1pt, Blue]
    \foreach \x in {0,...,2}
        \foreach \y in {0,...,1}
            \node[circle, fill=gray!80, inner sep=0pt, minimum size=3pt] at (\x,\y) {};
            
            \draw[-Stealth] (0, 0) -- (0, 1);
            \draw[-Stealth] (0, 0) -- (1, 0); 

            \node[left] at (0, 1) { \(e_0\) }; 
            \node[right] at (1, 0) { \(e'_1\) };            
    \end{tikzpicture}

    Doing the prescribed \(k_i\) blowups is equivalent to adding rays generated by \( (1, 1), (2, 1), \dots , (k_i, 1) \), 
    
    \begin{tikzpicture}[line width = 1pt, Blue]
    \foreach \x in {0,...,3}
        \foreach \y in {0,...,1}
            \node[circle, fill=gray!80, inner sep=0pt, minimum size=3pt] at (\x,\y) {};
            
            \draw[-Stealth] (0, 0) -- (0, 1);
            \draw[-Stealth] (0, 0) -- (3, 1); 
            \draw[-Stealth] (0, 0) -- (1, 1); 
            \draw[-Stealth] (0, 0) -- (2, 1);
            \draw[-Stealth] (0, 0) -- (1, 0);
 
            \node[left] at (0, 1) { \(e_0\) }; 
            \node[right] at (3, 1) { \(e_0 + k_i e'_1\) };   
            \node[right] at (1, 0) { \(e_1'\)};
    \end{tikzpicture}

    and deleting \(\widetilde{D_i}\) is the same as deleting the vector \(e_1'\). 
    Then we end up with the same toric fan. 
    Thus, \(\mathcal{B}_{\mathbf{\Sigma} } \) is isomorphic to the minimal resolution of \( |\mathbf{B}_\mathbf{\Sigma}| \). 
\end{proof}

\begin{corollary} \label{cor: nonstacky mirror}
    $\Coh (\mathcal{B}_{\mathbf{\Sigma}}) \cong \mathrm{Fuk}(\mathbf{A}_{\mathbf{\Sigma}})$.
\end{corollary}
\begin{proof}
    By \cite{Kapranov-Vasserot},    $\Coh(\mathbf{B}_{\mathbf{\Sigma}}) \cong \Coh(\mathcal{B}_{\mathbf{\Sigma}})$. 
\end{proof}

%% file: sections/B-side_toric_model.tex

\section{Toric models on the B-side}\label{sec:toric-model-on-B-side}

Recall from \cite{vanderput-saito} that the Painlevé character varieties come in families of complex affine cubic surfaces. For each Painlev\'e type $\tau$, we denote by \(\bM_{\tau}^{\un}\) the Painlevé character variety characterized by the parameter choices of \zcref{table:pain_params}. 
This choice of parameters is equivalent to asking that all topological monodromies around punctures are unipotent, and that the formal monodromies at irregular singularities are either the identity matrix or
{\tiny $
\begin{pmatrix}
0 & -1 \\
1 & 0
\end{pmatrix}
$}. 
This in turn is equivalent to asking that the microlocal monodromy along the Stokes Legendrian is trivial.

It turns out that, in the affine description of the Painlevé cubics, \(\bM_{\tau}^{\un}\) corresponds to the most singular member of the respective family.

Denote by \(\widetilde \bM_{\tau}^\un \) its minimal resolution.  We will also use $\tau$ to denote the stacky fan associated to each type in \zcref{fig:stacky-toric-models}.  Let \(\cB_\tau \) be as in Construction~\ref{construction:minimal_resolution}.  

Here we will prove: 

\begin{theorem} \label{thm: B precise}
    For each Painlev\'e type $\tau$, there is an isomorphism $\mathcal{B}_\tau \cong \widetilde \bM_{\tau}^\un $. 
\end{theorem}

\zcref{thm: mirror} now follows formally by combining the  result of \cite{beimler-olsen} that the Weinstein structures of the generic Painlev\'e character varieties are given by attaching along the Legendrians of \zcref{fig:painleve_handlebodies}; \zcref{thm:torus_skeleta} identifying this presentation with the toric presentation of Figure \ref{fig:torus_skeleta}; the mirror symmetry \zcref{cor: nonstacky mirror}; and finally the recharacterization of the mirrors in \zcref{thm: B precise}.  

\begin{remark}
    In the case of \(\PI, \PIIIDseven, \PIIIDeight \), the fiber $\bM_\tau^\un$ is smooth, hence equal to its resolution, and deformation equivalent to the generic fiber.  (In fact all fibers in the family are smooth.)  Thus our results assert that it is self-mirror. 
\end{remark}













\begin{table}[ht]
\centering

\begin{tabular}{l|c|p{6cm}|c}
\hline
\textbf{Painlevé type} & \textbf{Katz invariants} & \textbf{Parameters} & \textbf{Choice of parameters} \\
\hline
$PI$
& $(-,-,\frac52)$
& none
& \\
\hline

$PII$
& $(-,-,3)$
& $\alpha$ determines the formal monodromy $\gamma_\infty = D_\alpha$.
& $\alpha = 1$\\
\hline

$PII(FN)$
& $(0,-,\frac32)$
& $s =tr(M_0)$
& $s=2$\\
\hline

$PIII(D_6)$
& $(1,-,1)$
& $\alpha,\beta$ determine the formal monodromies $\gamma_0 = D_\alpha$, $\gamma_\infty = D_{\beta}$.  
& $\alpha = \beta = 1$\\
\hline

$PIII(D_7)$
& $(\frac12,-,1)$
& $\alpha$ determines the formal monodromy $\gamma_\infty = D_{\alpha}$.
& $\alpha = 1$\\
\hline

$PIII(D_8)$
& $(\frac12,-,\frac12)$
& none
& \\
\hline

$PIV$
& $(0,-,2)$
& $s_1 = tr(M_0)$, $s_2$ determines the formal monodromy $\gamma_\infty = D_{s_2}$.
& $s_1=2$, $s_2=1$ \\
\hline

$PV(deg)$
& $(0,0,\frac12)$
& $s_0 = tr(M_0),s_1 = tr(M_1)$.
& $s_0 = s_1 = 2$\\
\hline

$PV$
& $(0,0,1)$
& $s_1 = tr(M_0), s_2= tr(M_1)$, $s_3$ determines the formal monodromy $\gamma_\infty = D_{s_3}$.
& $(s_1, s_2, s_3) = (2,2,1)$\\
\hline

$PVI$
& $(0,0,0,0)$
& $a_i = tr(M_i)$ for $i \in \{1, \dots, 4\}$.
& $(a_1, a_2, a_3, a_4) = (2,2,2,2)$\\
\hline
\end{tabular}

\caption{Painlevé monodromy-space parameters and their local monodromy meaning.  Here, following \cite{vanderput-saito}, we denote a Painlevé type by a triple $(a,b,c)$ denoting the Katz invariant of the (possibly irregular) singularity at the points $p = 0, 1, \infty \in \P^1$.  We write $M_p$ for topological monodromy and $\gamma_p$ the formal monodromy. We write $D_\alpha := \mathrm{diag}(\alpha, \alpha^{-1})$.
}
\label{table:pain_params}
\end{table}

The remainder of this section gives the proof of \zcref{thm: B precise}. In each case we explicitly write down formulas for three regular functions \(x, y, z \in \cO(\mathcal{B}_\tau)\) satisfying the given equation for $\bM_{\tau}^\un$, hence defining a map $\mathcal{B}_\tau \to \bM_{\tau}^\un$. 
Once we have the expressions for these functions, 
checking that they are  indeed regular, and also that the resulting map is indeed a minimal resolution, is a direct calculation in coordinates. We record it below.

\subsection{Painlevé I}

We have 
\begin{equation}\label{eq:moduli_PI}
\bM_{\PI} \cong \{ xyz - x - z - 1  = 0\}.
\end{equation}
In this case \(\bM_\PI = \bM_\PI^\mathrm{un}\) is smooth. In fact, it is the \(A_2\)-cluster variety, and our regular functions will be cluster variables. 
The toric model is
\(X_{\PI} = \bA^1_u \times \bA^1_v\). 
Denote \( D_u = \{ u = 0 \} \), \(D_v = \{v =0 \}\). We blow up at \( (0, -1)\in D_u\) and \(  (-1, 0) \in D_v\).  Let \(\widetilde X_\PI\) be the resulting surface, \( \widetilde D_u  \) and \(\widetilde  D_v \) be the strict transforms, and \(E_u \) and \(E_v\) be the corresponding exceptional divisors. Then by our construction 
\[\cB_\PI = \widetilde X_\PI \setminus (\widetilde D_u \cup \widetilde D_v).
\]

Consider the rational functions
\begin{equation*}
    x = \frac{u+1}{v}, \qquad y = \frac{v+1}{u}, \qquad z = u.
\end{equation*}
They satisfy the equation~\eqref{eq:moduli_PI}, hence define a birational map
\begin{equation}\label{eq:birational_PI}
    \cB_\PI \dashrightarrow \bM_\PI. 
\end{equation}

We check that these functions are indeed regular. 
Since we remove \(\widetilde D_u\) and \(\widetilde D_v\) from \(\cB_\PI\), the only possible poles are on the exceptional divisors \(E_u\) and \(E_v\).

Let us compute in local coordinates. 
After blowing up at \( (-1, 0)\), there are two local charts: 
\begin{align*}
    &(a, v) \quad \text{where} \quad u = av - 1, \qquad
     \text{and}\\
    &(b, u) \quad \text{where} \quad v = b(u + 1),
\end{align*}
and locally \( \widetilde D_v \) is given by \( \{ b = 0\}\).  
On the first chart, the functions become
\[
x=a, \qquad y= \frac{v+1}{av-1}, \qquad z=av-1,
\]
and there are no poles around the exceptional divisor \(E_v\), locally given by \(\{v = 0\}\). On the second chart, the functions become
\[
x = \frac{1}{b}, \qquad y = \frac{1 + b + bu}{u}, \qquad z = u.
\]
Since \(u\) is invertible along the exceptional divisor, 
the only possible pole is \( b = 0\), which is on \(\widetilde D_v\), and removed from \( \cB_\PI\). Hence these functions are regular on \(\cB_\PI\) around the blowup at \( (-1, 0)\). 

Similarly, the functions are also regular on \(\cB_\PI\) around the blowup at \( (0, -1)\). 

Now we show that the map is a bijection. In fact we have the inverse map:
\[
u = z, \qquad v = yz  - 1. 
\]
We only need to take care of the loci \( \{u = 0\}\) and \( \{v = 0\}\). 

If \(v = 0\), then \( yz = 1\).  By the equation in~\eqref{eq:moduli_PI}, we get \( y = z = -1\).  By the above local coordinate calculation, the map \eqref{eq:birational_PI} sends the line \(E_v \cap \cB_{\PI}\) to  the line \(\{ y = z =-1 \} \). 

Similarly, if \(u = 0\), then \(z = 0\) and \(x = -1\). The map \eqref{eq:birational_PI} sends the line \(E_u \cap \cB_\PI \) to the line \( \{ x = -1, z = 0 \} \). Hence~\eqref{eq:birational_PI} is an isomorphism.

\subsection{Painlevé II/II(FN)} 
We have 
\begin{align*}
    \bM_\PII &\cong \{  xyz - x - \alpha y - z + \alpha + 1 = 0 \mid \alpha \in \bC^*\} \\
    \bM_\PIIFN
    &\cong \{ xyz - x - y - z+ s = 0 \mid s \in \bC \}. 
\end{align*}

For \(\bM_\PII\), the most singular fiber is over \(\alpha = 1\). For \(\bM_\PIIFN\), the most singular fiber is over \(s = 2\). 
In both cases, 
\begin{equation}\label{eq:moduli_PII_un}
 \bM_{\PII}^\un \cong \bM_{\PIIFN}^\un \cong \{ xyz - x - y - z +2 =0\},
\end{equation}
which has an \(A_1\) singularity at \( (x, y, z) = (1, 1, 1)\). (For \(\PIIFN\) we replace \(x, z \) by \(-x, -z\).)

The toric model is  
\(X_{\PII/\PIIFN} = \bP^2_{u, v, w}\).  Denote the boundary divisors by
\(
D_u=\{u=0\}\), \( D_v=\{v=0\}\), and \( D_w=\{w=0\}\). 
We blow up at the three points
\[
p_u=[0:1:-1],\qquad 
p_v=[1:0:-1],\qquad 
p_w=[1:-1:0].
\]
Let \(\widetilde X_{\PII/\PIIFN}\) be the resulting surface, and denote by
\(\widetilde{D}_u,\widetilde{D}_v,\widetilde{D}_w\) the strict transforms of
the boundary divisors. Then by our construction
\[
\cB_{\PII/\PIIFN}
=
\widetilde{X}_{\PII/\PIIFN}
\setminus
(\widetilde{D}_u\cup \widetilde{D}_v\cup \widetilde{D}_w).
\]

Consider the rational functions
\begin{equation}\label{eq:regular_functions_for_PII}
    x=-\frac{u+v}{w},\qquad
    y=-\frac{u+w}{v},\qquad
    z=-\frac{v+w}{u}.
\end{equation}
They satisfy the equation in~\eqref{eq:moduli_PII_un}, 
hence define a rational map
\begin{equation}\label{eq:resolution_PII}
    \cB_{\PII/\PIIFN}\dashrightarrow \bM_{\PII/\PIIFN}^{\mathrm{un}}.
\end{equation}

We check regularity near \(p_u=[0:1:-1]\); the other two cases are cyclically identical.
Take the affine chart \(v=1\), with coordinates \((u,w)\). Then \(p_u=(0,-1)\), and
\[
x=-\frac{u+1}{w},\qquad
y=-(u+w),\qquad
z=-\frac{1+w}{u}.
\]
After blowing up \(p_u\), consider the chart
\[
w+1=au.
\]
Then
\[
x=-\frac{u+1}{au-1},\qquad
y=1-(a+1)u,\qquad
z=-a,
\]
so there are no poles on the exceptional divisor \( \{ u = 0\}\). On the other chart
\[
u=b(w+1),
\]
we have
\[
x=-\frac{b(w+1)+1}{w},\qquad
y=-(b(w+1)+w),\qquad
z=-\frac{1}{b}.
\]
The only possible pole is along \(b=0\), which is on the strict transform
\(\widetilde{D}_u\), hence removed from \(\cB_{\PII/\PIIFN}\).
Thus the functions are regular near \(p_u\). By cyclic symmetry, they are regular near
\(p_v\) and \(p_w\) as well. Therefore~\eqref{eq:resolution_PII} is a morphism.

Next, consider the line
\[
L=\{u+v+w=0\}\subset \bP^2.
\]
It passes through \(p_u,p_v,p_w\). Along \(L\), we have
\[
x=y=z= 1.
\]
Hence the strict transform of \(L\), which we denote by \(\widetilde L\), is contracted to the
singular point \((1,1,1)\).

Away from the contracted curve, the inverse map is given by solving the above linear
system. For example, wherever the following expression is nonzero, we have
\[
[u:v:w]=[1-xy:x-1:y-1].
\]
Cyclically, one obtains the corresponding expressions on the remaining charts. Thus the
morphism~\eqref{eq:resolution_PII} is an isomorphism away from \(\widetilde L\).  Finally,
\(\widetilde L\) was a line in \(\bP^2\) on which we blew up three points, so $\widetilde L^2 - 1-3= -2$. Since the surface~\eqref{eq:moduli_PII_un} has an
\(A_1\)-singularity at \((1,1,1)\), the morphism~\eqref{eq:resolution_PII}
is the minimal resolution.

\subsection{Painlevé III(D6)/V(deg)}  We have
\begin{align*}
\bM_{\PIIIDsix} &= \{ xyz + x^2 + y^2 + (1 + \alpha \beta) x + (\alpha + \beta)y + \alpha \beta = 0 \mid \alpha ,\beta \in \bC^* \}\\
\bM_{\PVdeg} &= \{ xyz + x^2 + y^2 + s_0 x + s_1 y + 1= 0 \mid s_0, s_1 \in \bC \}
\end{align*}
For \(\PIIIDsix\), the singular fibers lie over \( \alpha = \beta \) and \(\alpha \beta = 1\). 
We consider the fiber over \( \alpha = \beta = 1\). 

For \(\PVdeg\), the singular fibers lie over \(s_0 = \pm 2\) and \(s_1 = \pm 2\). We consider the fiber over \( s_0 = s_1 = 2\). Hence
\begin{equation} \label{eq:eq_MPIIIDsix_un}
    \bM_{\PIIIDsix}^\mathrm{un} \cong \bM_{\PVdeg}^\mathrm{un} \cong \{ xyz + x^2 + y^2 + 2x + 2y + 1 = 0 \}.
\end{equation}
There are two \(A_1\) singularities at \((x, y, z) = (0, -1, 2)\) and \( (x, y, z) = (-1, 0, 2)\). 

The toric model is \( X_{\PIIIDsix / \PVdeg} \cong \bA^1_u  \times \bA^1_v \), and we blow up \( (u, v) = (-1, 0) \) and \((u, v) = (0, -1)\), twice at each point. Denote the resulting surface by \(\widetilde X_{\PIIIDsix/\PVdeg} \), and the strict transform of \(D_u = \{u=0\} \) and \(D_v = \{v= 0 \}\) by \( \widetilde D_u\) and \(\widetilde D_v\). Then 
\[
\cB_{\PIIIDsix/\PVdeg} = \widetilde X_{\PIIIDsix/\PVdeg} \setminus (\widetilde D_u\cup \widetilde D_v). 
\]

The regular functions are \begin{equation*}
    x = u, \qquad y = v, \qquad z = - \frac{u^2 + v^2 + 2u + 2v + 1}{uv}
\end{equation*}
which give rise to a rational map
\begin{equation}\label{eq:resolution_PIIIDsix}
\cB_{\PIIIDsix/\PVdeg} \dashrightarrow \bM_{\PIIIDsix/\PVdeg}^\un. 
\end{equation}

First, let us show the regularity of these functions. We only need to consider the local coordinates around \( (-1, 0) \in D_v\). The case for \( (0, -1)\) is symmetric.  After the first blowup, there are two charts:
\begin{align*}
    u+1 = av \qquad \text{and} \qquad v = b (u+1).
\end{align*}
On the first chart we have
\[
x = av -1 , \qquad y =v, \qquad z = - \frac{(a^2+1)v + 2}{av-1}
\]
so there are no poles around the exceptional divisor \( \{ v = 0 \}\). 
On the second chart we have
\[
x = u, \qquad y = b(u+1), \qquad z = -\frac{ (u+1) + b^2 (u+1) + 2b}{b u},
\]
and the strict transform of \(D_v\) coincides with \( \{b = 0\}\). 
Then we perform the second blowup at \( (u, b) = (-1, 0)\). After the second blowup, there are also two charts:
\[
 u+1 = db, \qquad \text{and} \qquad b = c (u+1). 
\]
On the first chart we have 
\[
x = db -1, \qquad y= b^2 d, \qquad  z = -\frac{b^2 d + d +2}{db-1}
\]
so there are no poles around the exceptional divisor \( \{ b = 0\}\). On the second chart we have
\[
x = u, \qquad y = c(u+1)^2, \qquad z = - \frac{c^2 (u+1)^2 + 2c +1 }{cu}
\]
where the only pole at \( c = 0\) is on \(\widetilde D_v\), which is removed. 

Note that the exceptional divisor of the first blowup is mapped to the singular point \( (-1, 0, 2) \), and the exceptional divisor of the second blowup is mapped to the line \( \{x = -1, y= 0\}\).

Similarly, if we look at the blowups performed at \((0, -1) \in D_u\), the first exceptional divisor is mapped to the singular point \( (0, -1, 2)\), and the second  to the line \( \{ x = 0, y =-1 \}\).

From the defining equation~\eqref{eq:eq_MPIIIDsix_un}, we see that \(\bM_{\PIIIDsix/\PVdeg}^\un\) is stratified by the torus \( (\bC^*)^2_{x, y} \) together with the two lines \( \{x = -1, y= 0\}\) and \( \{ x = 0, y =-1 \}\). 
Hence \eqref{eq:resolution_PIIIDsix} is a bijection everywhere except the singular points. 

The two curves contracted to the two \(A_1\)-singularities both have self-intersection number \(-2\). 
Combining all the arguments above, we have shown that \eqref{eq:resolution_PIIIDsix} is the minimal resolution.

\subsection{Painlevé III(D7)}We have
\[
 \bM_{\PIIIDseven} = \{ xyz + x^2 + y^2 + \alpha x + y = 0 \mid \alpha \in \bC^* \}.
\]
Here all fibers are smooth. We consider the fiber over \(\alpha = 1\). Hence
\begin{equation}\label{eq:moduli_PIIID7}
 \bM_{\PIIIDseven}^\mathrm{un} = \{ xyz + x^2 + y^2 + x + y = 0\}.
\end{equation}

The toric model \(X_{\PIIIDseven}\) is given by the fan  \begin{tikzpicture}[Blue, scale = .5, line width = .5pt]
    \fill[orange, opacity = .3] (0, 0) -- (1, 0) -- (1, 1);
    \fill[Green, opacity = .3] (0, 0) -- (0, 1) -- (1, 1);
    \draw[-Stealth] (0,0) -- (1, 0);
    \draw[-Stealth] (0, 0) -- (0, 1); 
    \draw[-Stealth] (0, 0) -- (1, 1); 
    
\end{tikzpicture}. 
In other words, we take \(\bA_u^1 \times \bA_v^1\) and blow up at \( (u, v) = (0, 0)\). Denote the exceptional divisor by \(D_e\). Then we blow up three points: \( (0, -1)\) on \(D_u = \{ u = 0\}\), \( (-1, 0)\) on \(D_v = \{ v = 0\} \), and \( \frac{u}{v} = -1\) on  \(D_e\). Let \(\widetilde X\) be the resulting surface, and \(\widetilde D_u\), \(\widetilde D_v\), \(\widetilde D_e\) be the strict transforms. Then
\[
\cB_\PIIIDseven = \widetilde X_\PIIIDseven \setminus ( \widetilde D_u \cup \widetilde D_v \cup \widetilde D_e).
\]

Consider the rational functions
\begin{equation*}
    x = u , \qquad y = {v}, \qquad z = -\frac{u}{v} - \frac{v}{u} - \frac{1}{u} - \frac{1}{v}.
\end{equation*}
They define a birational map
\begin{equation}\label{eq:resolution_PIIID7}
\cB_{\PIIIDseven} \dashrightarrow \bM_{\PIIIDseven}^\un. 
\end{equation}

Let us first check the regularity around \( (u, v) = (-1, 0)\). After blowup, first consider the chart 
\[
 u+ 1 = a v,
\]
where 
\[
 x = av -1, \qquad y = v, \qquad z = - \frac{1 - a + v + a^2 v}{av - 1}. 
\]
There are no poles on the exceptional divisor \( \{ v = 0 \} \). On the chart 
\[
 v = b (u+1), 
\]
we have 
\[
x = u, \qquad y = b(u+1), \qquad z = - \frac{1}{b} - \frac{b(u+1)+1}{ u}. 
\]
The only possible pole on the exceptional divisor \( \{u = - 1\}\) is \( b = 0\), which is on \(\widetilde D_v\), hence removed. 

The case for the blowup at \( (u, v) = (0, -1)\) is symmetric. 

Now we consider the blowup on \(D_e\). Take the chart:
\[
u = rv. 
\]
Then
\[
x = rv, \qquad y = v, \qquad z = -r - \frac{1}{r} - \frac{r+1}{rv}. 
\]

The point on which we perform the blowup is \( (r, v) = (-1, 0) \). After the non-toric blowup, first consider the chart
\[
 r+1 = av 
\]
on which
\[
x = (av-1)v, \qquad y = v, \qquad z = 1-av - \frac{a+1}{av-1}. 
\]
There are no poles near the exceptional divisor \( \{v = 0\}\). On the chart
\[
v = b(r+1), 
\]
we have 
\[
x = rb(r+1) , \qquad y = b(r+1), \qquad z = -r - \frac{1}{r} - \frac{1}{rb}. 
\]
The only possible pole on the exceptional divisor \( \{ r = -1  \}\) is \( b = 0\), which is on the strict transform of the divisor \(\widetilde D_e\), hence removed. 

Then we check that \eqref{eq:resolution_PIIID7} is a bijection. From the defining equation~\eqref{eq:moduli_PIIID7}, one observes that \(\bM_{\PIIIDseven}^\un\) has a stratification:
\[
\bM_\PIIIDseven^\un \cong (\bC^*)^2_{x, y} \sqcup \{ x = 0, y = -1\} \sqcup \{ x = -1, y = 0 \} \sqcup \{ x = 0, y = 0\}. 
\]
By the local coordinate computation above, the morphism~\eqref{eq:resolution_PIIID7} sends the three exceptional divisors to the three lines \( \{ x= 0, y= -1\}\), \( \{x = -1, y=0\}\), and \( \{x = 0, y= 0\}\). Hence it is an isomorphism.

\subsection{Painlevé III (D8)}
We have
\[
 \bM_{\PIIIDeight} = \{xyz + x^2 +y^2 + y = 0 \}.
\]
The toric model \(X_{\PIIIDeight}\) is given by the fan
\begin{tikzpicture}[scale = .5]
    \fill[Green, opacity = .3] (1, 0) -- (1, 2) -- (0, 0); 
    \draw[-Stealth] (0,0) -- (1, 0);
    \draw[-Stealth] (0, 0) -- (1, 2); 
    \fill (1, 1) circle (2pt);
\end{tikzpicture},
i.e. it is the quadratic cone \(\{ ac = b^2\}\). Then we blow up the two points \( (a, b, c)  = (-1, 0, 0)\) and \( (a, b, c) = (0, 0, -1)\). Denote the resulting surface by \(\widetilde X_\PIIIDeight\), the toric divisors by \(D_a = \{a = 0\}  \) and \( D_c = \{ c =0\}\), and their strict transform by \(\widetilde D_a \) and \(\widetilde D_c\). Then
\[
\cB_\PIIIDeight = \widetilde X_\PIIIDeight \setminus ( \widetilde D_a \cup \widetilde D_c ). 
\]

The rational functions 
\begin{equation*}
    x = b, \qquad
    y=  a, \qquad
    z = -\frac{a+c+1}{b}
\end{equation*}
define a map
\begin{equation}
    \label{eq:resolution_PIIID8}
    \cB_\PIIIDeight \dashrightarrow \bM_\PIIIDeight. 
\end{equation}

Let us check the regularity around \( (a, b) = (-1, 0) \). On the chart
\[
a + 1 = ub,
\]
we have 
\[
x = b, \qquad y = a, \qquad z = - \frac{ u (ub-1) + b }{ub - 1}.
\]
There are no poles. On the chart
\[
b = v(a+1)
\]
we have
\[
x = v(a+1), \qquad y = a, \qquad z = -\frac{a + v^2(a+1)}{av}.
\]
The only possible pole is \( v = 0\), which is on \(\widetilde D_c\), hence removed. 

The regularity around \( (c, b) = (-1, 0) \) is similar.

Now we show that the map~\eqref{eq:resolution_PIIID8} is a bijection. The space \(\bM_\PIIIDeight\) is stratified as:
\[
\bM_\PIIIDeight \cong (\bC^*)^2_{x, y} \sqcup \{ x =0, y=0\} \sqcup \{ x = 0, y = -1\}. 
\]
The morphism~\eqref{eq:resolution_PIIID8} sends the two exceptional divisors to the two lines \( \{ x = 0 , y = 0\}\) and \(\{ x = 0, y = -1\}\), hence it is bijective.

\subsection{Painlevé IV} We have
\[ 
\bM_{\PIV} = \{ xyz + x^2 - (s_2^2 + s_1 s_2) x - s_2^2 y -s_2^2 z + s_2^2 + s_1 s_2^3 = 0 \mid s_1 \in \bC, s_2 \in \bC^* \} 
\]
The singular fibers lie over \( \{s_1 = \pm 2\}\) and \( \{s_1 = s_2 + s_2^{-1} \}\).  We consider the most singular fiber at \( (s_1 , s_2) = (2, 1)\). Hence
\[
\bM_{\PIV}^\mathrm{un} = \{ xyz + x^2 -3 x - y - z + 3 =0 \}.
\]
There is an \(A_2\) singularity at \( (x, y, z) = (1, 1, 1)\). 

The toric model is
\(\bA^1_u \times \bP^1_v\). We blow up once at  \( (u, v) = (0, -1) \) and \( (u, v) = (-1, 0)\), and twice at \( (-1, \infty)\). Denote \(t = \frac{1}{v}\), the toric divisors by \(D_u = \{ u = 0\}\), \( D_v = \{v = 0\}\), \(D_t = \{t = 0\}\), their strict transforms by \( \widetilde D_u\), \( \widetilde D_v \), \(\widetilde D_t\), the blown up surface by \(\widetilde X_\PIV\), so
\[
\cB_{\PIV} = \widetilde X_\PIV \setminus ( \widetilde D_u \cup \widetilde D_v \cup \widetilde  D_t).
\]
The rational functions 
\begin{equation*}
    x = -u, \qquad 
    y = -\frac{ (u+1)^2 v + 1}{u}, \qquad 
    z = -\frac{u + v +1}{ uv}
\end{equation*}
define a map
\begin{equation}
    \label{eq:resolution_PIV}
    \cB_\PIV \dashrightarrow \bM_\PIV^\un.
\end{equation}

First, let us check the regularity around \( ( u, v) = (0, -1)\). On the chart
\[
v + 1 = au, 
\]
we have 
\[
x = -u, \qquad y = u + 2 - a(u+1)^2, \qquad z = - \frac{a+1}{au - 1}. 
\]
In this chart, the exceptional divisor \( E_u\) coincides with \( \{ u = 0\} \), on which there are no poles. Under~\eqref{eq:resolution_PIV}, \(E_u\) is mapped to the line \( \{ x = 0, y+z = 3\}\). On the other chart 
\[
u = b(v+1)
\]
we have
\[
x = -b(v+1), \qquad y = - \frac{b^2 v^2+b^2v + 2 bv +1}{b}, \qquad z = -\frac{b+1}{bv}.
\]
In this chart, \(E_u\) coincides with \( \{ v = -1\}\), on which the only possible pole is \( (b, v) = (0, -1)\). This point is also on \(\widetilde D_u\), hence removed. 

Next, consider the blowup at \((-1, 0)\). On the chart
\[
u + 1 = av
\]
we have 
\[
x = 1 - av, \qquad y = - \frac{a^2 v^3 + 1}{av - 1}, \qquad z = -\frac{a+1}{av - 1}.
\]
In this chart the exceptional divisor \(E_v\) coincides with \( \{v = 0\}\), on which there are no poles. Under the map~\eqref{eq:resolution_PIV}, \(E_v\) is sent to the line \( \{x = 1, y=1\}\). On the chart 
\[
v = b(u+1)
\]
we have 
\[
x = -u, \qquad y = -\frac{b(u+1)^3 + 1}{u}, \qquad z = -\frac{b+1}{bu}. 
\]
The only possible pole is \( (b, u) = (0, -1)\), which is on \(\widetilde D_v\), hence removed. 

Finally, we consider the blowups at \( (u, t) = (-1, 0)\). We blow up once and consider the chart:
\[
u+1 = at.
\]
We have 
\[
x = 1- at, \qquad y = -\frac{a^2 t +1 }{at-1}, \qquad z = -\frac{at^2+1}{at - 1}. 
\]
The exceptional divisor \(E_{t, 1}\) coincides with \(\{t = 0\}\), on which there are no poles. Under~\eqref{eq:resolution_PIV}, \(E_{t, 1}\) is sent to the point \((1, 1, 1)\). On the other chart
\[
t = b(u+1), 
\]
we have 
\[
x = -u, \qquad y = -\frac{b+u+1}{bu}, \qquad z = -\frac{b(u+1)^2 +1}{ u }.
\]
We perform the second blowup at  \( (b, u) = (0, -1)\) and consider the chart
\[
u+1 = db,
\]
on which
\[
x = 1 -db, \qquad y = -\frac{d+1}{db - 1}, \qquad z = -\frac{d^2 b^3 + 1}{db -1}. 
\]
The exceptional divisor \(E_{t, 2}\) in this chart coincides with \( \{ b = 0\}\). There are no poles on it. Under~\eqref{eq:resolution_PIV}, \(E_{t, 2}\) is mapped to the line \( \{x = 1, z = 1\}\). On the other chart
\[
 b = c(u+1), 
\]
we have 
\[
x= -u, \qquad y = -\frac{c+1}{cu}, \qquad z = -\frac{c(u+1)^3+1}{u}.
\]
The only possible pole is at \( (c, u) = (0, -1)\), which is on \(\widetilde D_t\), hence removed. 

Now we show that the map \eqref{eq:resolution_PIV} is surjective. Indeed, we can write the inverse rational map:
\[
u = -x, \qquad v = \frac{xy - 1}{(1-  x)^2},
\]
which is defined except for the loci \( \{x = 0\} \) and \( \{x = 1\}\). If \(x = 0 \), then \( y+ z= 3\). This line is identified with \( E_u \). If \(x = 1\), then \(  (y-1) (z-1) = 0\). These two lines are covered by \( E_v\) and \(E_{t, 2}\). Hence~\eqref{eq:resolution_PIV} is surjective.

Let \( L = \{ u = -1\}\), and denote its strict transform by \(\widetilde L\). Along \(\widetilde L\), we have \(x = y = z = 1\). Hence the map contracts the curves \(\widetilde L\) and \(E_{t, 1}\). They are both \(-2\)-curves, and intersect transversally at one point. The above local calculations show that~\eqref{eq:resolution_PIV} is injective away from \(\widetilde L\) and \(E_{t, 1}\). Combining everything above, we see that~\eqref{eq:resolution_PIV} is a minimal resolution.

\subsection{Painlevé V}
We have
\[
\bM_{\PV} = \{ xyz + x^2 + y^2 - (s_1 + s_2 s_3) x - (s_2 + s_1 s_3) y - s_3 z + s_3^2 +  s_1 s_2 s_3 +1 = 0 \mid s_1, s_2 \in \bC, s_3 \in \bC^*\}. 
\]

This is singular when \( \{ s_1 = \pm 2 \}\), \( \{s_2 = \pm 2\}\), and \( \big\{(s_1^2+ s_2^2) s_3^2 - s_1 s_2 s_3 (s_3^2 + 1) + (s_3^2 - 1)^2 = 0 \big\}\). 
We consider the most singular fiber at \((2, 2, 1)\):
\[
\bM_{\PV}^\mathrm{un} = \{ xyz + x^2 + y^2 - 4x - 4y - z + 6 = 0 \}.
\]
There is an \(A_3\) singularity at \( (x, y, z) = (1, 1, 2)\). 

We take the toric model \(X_{\PV}=\bA^1_u \times \bP^1_v\), and blow up once at \( (u, v) = (0, -1)\), and twice at each of the points \( (u, v) = (-1, 0)\) and \( (u, v) = (-1, \infty)\). Denote \(t=1/v\), the toric divisors by \(D_u=\{u=0\}\), \(D_v=\{v=0\}\), \(D_t=\{t=0\}\), their strict transforms by \(\widetilde D_u,\widetilde D_v,\widetilde D_t\), and the surface after the blowups by \(\widetilde X_{\PV}\). Then
\[
\cB_{\PV}
=
\widetilde X_{\PV}
\setminus
(\widetilde D_u\cup \widetilde D_v\cup \widetilde D_t).
\]

The rational functions
\[
x=-u,\qquad
y=-\frac{(u+1)^2+v}{uv},\qquad
z=-\frac{(u+1)^2(v^2+1)+2v(2u+1)}{u^2v}
\]
satisfy the equation of \(\bM_{\PV}^{\mathrm{un}}\), hence define a rational map
\begin{equation}\label{eq:resolution_PV}
    \cB_{\PV}\dashrightarrow \bM_{\PV}^{\mathrm{un}}.
\end{equation}

First, let us check regularity around \((u,v)=(0,-1)\). On the chart
\[
v+1=au,
\]
we have
\[
x=-u,\qquad
y=-\frac{a+u+2}{au-1},\qquad
z=-\frac{a^2u^2+2a^2u+a^2-2au+2}{au-1}.
\]
There are no poles around the exceptional divisor \(E_u=\{u=0\}\). Under the map~\eqref{eq:resolution_PV}, \(E_u\) is mapped to the curve
\[
\{x=0,\ z=y^2-4y+6\}.
\]
Indeed, on \(E_u\), we have \((x,y,z)=(0,a+2,a^2+2)\).

On the other chart
\[
u=b(v+1),
\]
we have
\[
x=-b(v+1),\qquad
y=-\frac{b^2v+b^2+2b+1}{bv}, \qquad
z=-\frac{b^2v^2+b^2+2bv+2b+1}{b^2v}.
\]
The only possible pole lies along \(b=0\), which is the strict transform \(\widetilde D_u\), hence removed.

Next, consider the blowups at \((-1,0)\). After the first blowup, on the chart
\[
u+1=av,
\]
we have
\[
x=1-av,\qquad
y=-\frac{a^2v+1}{av-1},\qquad
z=-\frac{a^2v^3+a^2v+4av-2}{(av-1)^2}.
\]
There are no poles around the exceptional divisor \(E_{v,1}=\{v=0\}\). Under the map~\eqref{eq:resolution_PV}, \(E_{v,1}\) is sent to the point \((1,1,2)\).

On the other chart
\[
v=b(u+1),
\]
we have
\[
x=-u,\qquad
y=-\frac{b+u+1}{bu},
\]
and
\[
z=-\frac{b^2u^3+3b^2u^2+3b^2u+b^2+4bu+2b+u+1}{bu^2}.
\]
Then we perform the second blowup at \((b,u)=(0,-1)\). On the chart
\[
u+1=db,
\]
we have
\[
x=1-db,\qquad
y=-\frac{d+1}{bd-1},\qquad
z=-\frac{b^4d^3+4bd+d-2}{(bd-1)^2}.
\]
The exceptional divisor \(E_{v,2}\) is given by \(\{b=0\}\). The map~\eqref{eq:resolution_PV} sends \(E_{v,2}\)  to the line
\[
\{x=1,\ y+z=3\}.
\]

On the other chart
\[
b=c(u+1),
\]
we have
\[
x=-u,\qquad
y=-\frac{c+1}{cu},
\]
and
\[
z=-\frac{c^2u^4+4c^2u^3+6c^2u^2+4c^2u+c^2+4cu+2c+1}{cu^2}.
\]
Since \(u\) is invertible near \(u=-1\), the only possible pole lies along \(c=0\), which is the strict transform \(\widetilde D_v\), hence removed.

The calculation near \((-1,\infty)\) follows from the calculation near \((-1,0)\) by the involution \(v\mapsto v^{-1}\). This involution preserves \(x\) and \(z\), and sends \(y\) to \(4-xz-y\). It exchanges the two centers \((-1,0)\) and \((-1,\infty)\). Hence the first exceptional divisor over \((-1,\infty)\), denoted by \(E_{t,1}\), is also contracted to \((1,1,2)\), while the second exceptional divisor \(E_{t,2}\) maps isomorphically onto the line
\[
\{x=1,\ y=1\}.
\]
Potential poles lie along \(\widetilde D_t\), hence are removed. Therefore~\eqref{eq:resolution_PV} is a morphism.

Now we show that the map~\eqref{eq:resolution_PV} is surjective. Away from the exceptional loci, we have the inverse rational map
\[
u=-x,\qquad
v=\frac{(1-x)^2}{xy-1}.
\]
If \(x=0\), then the defining equation gives \(z=y^2-4y+6\), and this curve is covered by \(E_u\). If \(x=1\), then the defining equation gives
\[
(y-1)(y+z-3)=0.
\]
The line \(\{x=1,\ y+z=3\}\) is covered by \(E_{v,2}\), and the line \(\{x=1,\ y=1\}\) is covered by \(E_{t,2}\). Hence~\eqref{eq:resolution_PV} is surjective.

Let \(L=\{u=-1\}\), and denote its strict transform by \(\widetilde L\). Along \(\widetilde L\), we have \(x=1\), \(y=1\), and \(z=2\). Thus \(\widetilde L\) is contracted to the singular point \((1,1,2)\). The strict transforms of \(E_{v,1}\) and \(E_{t,1}\) are also contracted to the same point. These three curves are \((-2)\)-curves: the curve \(L\) has self-intersection \(0\) before the blowups and passes through the two points \((-1,0)\) and \((-1,\infty)\), while \(E_{v,1}\) and \(E_{t,1}\) are the first exceptional divisors blown up once more. Moreover, they form the \(A_3\)-chain
\[
E_{v,1}\;-\;\widetilde L\;-\;E_{t,1}.
\]
The local calculations above show that the morphism is injective away from these three curves. Therefore~\eqref{eq:resolution_PV} is the minimal resolution of \(\bM_{\PV}^{\mathrm{un}}\).

\subsection{Painlevé VI} We have
\begin{align*}
\bM_{\PVI}
&\cong \Bigl\{
xyz + x^2 + y^2 + z^2
- s_1 x - s_2 y - s_3 z + s_4 = 0 \\
&\qquad \mid\;
s_1 = a_1 a_4 + a_2 a_3,\;
s_2 = a_2 a_4 + a_1 a_3,\;
s_3 = a_3 a_4 + a_1 a_2, \\
&\qquad\qquad
s_4 = a_1 a_2 a_3 a_4
+ a_1^2 + a_2^2 + a_3^2 + a_4^2 - 4,\;
a_i \in \mathbb{C}
\Bigr\}.
\end{align*}
We consider the fiber over \( (a_1, a_2, a_3, a_4 ) = (2, 2, 2, 2)\).
It is the most singular fiber, and also corresponds to unipotent monodromies around each puncture (see \cite[section 9]{Inaba-Iwasaki-Saito}, for example). The defining equation becomes:
\[
\{ x^2 + y^2 + z^2 + xyz - 8x - 8y - 8z + 28 = 0\} 
\]
By replacing \(x, y, z\) by \(x+2, y+2, z+2\), we get
\begin{equation}\label{eq:singular-PVI}
 \bM_{\PVI}^\un \cong \{ (x+y+z)^2 + xyz = 0\}.
\end{equation}
There is a \(D_4\) singularity at the origin. We take the toric model to be \(X_{\PVI} = \bP^2_{u_1, u_2, u_3}\). Here the subscript means that we write the coordinates  as \([u_1: u_2: u_3]\).

The boundary divisors are 
\[
D_i = \{ u_i = 0 \}, \quad i = 1, 2, 3.
\]
We blow up twice at the points
\[
p_1=[0:1:-1]\in D_1,\qquad 
p_2=[1:0:-1]\in D_2,\qquad 
p_3=[1:-1:0]\in D_3.
\]
Note that these three points pass through a line \(L = \{ u_1 + u_2 + u_3 = 0\}\). 

More precisely, 
we first blow up at the points \(\{p_i\}\), 
and denote the exceptional divisors by \(E_i\), and the strict transform of \(D_i\) by \(D_i'\). Then we
further blow up the points \(\{D_i'\cap E_i\}\).
Let \(\widetilde X_\PVI\) be the surface after the blowups, \(\widetilde  D_\PVI  = \cup_i \widetilde D_i \) the strict transform of the toric boundary divisor, then
\[
\cB_\PVI = \widetilde X_\PVI \setminus \widetilde D_\PVI. 
\]

Consider the rational functions on \(\cB_\PVI\):
\begin{equation}\label{eq:regular_functions_for_PVI}
    x = -\frac{(u_1+u_2+u_3)^2}{u_2 u_3}, \qquad y =- \frac{(u_1+u_2+u_3)^2}{u_1 u_3}, \qquad z = -\frac{(u_1+u_2+u_3)^2}{u_1 u_2}. 
\end{equation}
They satisfy the equation in~\eqref{eq:singular-PVI}, hence define a map
\begin{equation}\label{eq:resolution of PVI}
 \cB_\PVI \dashrightarrow \bM_{\PVI}^\mathrm{un}. 
\end{equation}

We check  regularity near \(p_1=[0:1:-1]\); the other two points are cyclically identical.
We write the regular functions in~\eqref{eq:regular_functions_for_PVI}
in coordinates \( (u_1, u_3)\) on the affine chart \(u_2=1\):
\[
x= - \frac{(1+u_1 + u_3)^2}{u_3},\qquad
y= - \frac{(1+u_1 + u_3)^2}{u_1 u_3},\qquad
z= - \frac{(1+u_1 + u_3)^2}{u_1}.
\]
After the first blowup at \( (u_1, u_3) = ( 0, -1)\), take two local charts: 
\begin{align*}
    &(a, u_3), \quad \text{where} \quad u_1 = a (u_3 + 1), \\
    &(b, u_1), \quad \text{where} \quad bu_1 = u_3 + 1.
\end{align*}
On the second chart, we have 
\[
x = -\frac{ u_1^2 (1 + b)^2}{ b u_1  -1  }, \qquad
y = - \frac{u_1 (1 + b)^2}{  b u_1 -1}, \qquad
z = - u_1 (1 + b)^2.
\]
There are no poles near \( (b, u_1) = (0, 0) \). So we only need to consider the chart \(u_1 = a (u_3 + 1)\). The divisor \(D_1'= \{ a = 0\} \), and 
\[
x = -\frac{(a+1)^2 (u_3 + 1)^2}{u_3}, \qquad
y = - \frac{ (a+1)^2 (u_3 +1) }{a u_3}, \qquad
z = - \frac{ (a+1)^2 (u_3 +1) }{a }. 
\]
Then we take the second blowup at the point \( (a, u_3) = (0, -1)\). Similarly, one can check that the functions \(x, y, z\) are regular on the chart \(ad = u_3 + 1\). Hence we only need to consider the chart 
\( a = (u_3 + 1) c\), where
\[
x = -\frac{(c(u_3+1) +1)^2 (u_3 + 1)^2}{u_3}, \qquad
y = -\frac{ ((u_3 + 1) c +1)^2 }{c u_3}, \qquad
z = -\frac{ ((u_3 + 1) c+1)^2 }{c}. 
\]
The only poles lie along \( \widetilde D_1 = \{ c = 0\} \), which are removed from \(\cB_\PVI\).  Hence we have shown that the functions \(x, y, z\) are indeed regular. 

We check \eqref{eq:resolution of PVI} is surjective. Consider \( (x, y, z) \in \bM_{\PVI}^\un \). If \(xyz\neq 0\), then this point lies in the image of the open torus of \(\cB_\PVI\). If \(x=0\) and \(yz\neq 0\), then the defining equation implies \(y+z=0\), and this point lies in the image of the second exceptional divisor over \(p_1\). The cases \(y=0\) and \(z=0\) are similar. Finally, the origin \((0,0,0)\) lies in the image of \(\widetilde L \) and \(E_i\), where \(\widetilde L\) is the stric transform of \(L\).

One can also see from the above local coordinates that the map~\eqref{eq:resolution of PVI} contracts the curves \(\widetilde L\) and \(E_i\), and is injective away from them. 
We also have
\[
\widetilde L^2=1-3=-2,\qquad E_i^2=-2,
\]
and \(\widetilde L\) meets each \(E_i\) transversely once. Hence they form the \(D_4\) Dynkin configuration. Therefore~\eqref{eq:resolution of PVI} is the minimal resolution.

%% file: sections/degenerations.tex

\section{Other Painlev\'e types from degenerations of \(\PVI\)}


\begin{proposition}
    The generic character variety of each Painlev\'e type can be obtained from the generic character variety of \(\PVI\) by removing Weinstein handles. More precisely, we have the following diagram, where each arrow means removing a Weinstein \(2\)-handle. 
    \begin{equation}\label{eq:degeneration}
    \begin{tikzcd}[column sep=1em, row sep=1em]
    && \PVdeg/\PIIIDsix \ar[r] \ar[rdd] & \PIIIDseven \ar[r] \ar[rdd] & \PIIIDeight \\
    \PVI \ar[r] & \PV \ar[ru] \ar[rd]\\
    && \PIV \ar[r] & \PII/\PIIFN \ar[r]& \PI
    \end{tikzcd}
    \end{equation}
    This diagram matches the confluence scheme for the Painlev\'e equations; see e.g. ~\cite{CMR}. 
    
\end{proposition}

\begin{proof}
    We use the mutations described in \cite{STW} and \( \mathrm{GL}(2, \bZ)\)-transformations to transform the presentations in Figure~\ref{fig:torus_skeleta}, from which the result is immediate: 
    
    For \(\bM_\PV^\gen\), consider:
    \[
    \begin{tikzpicture}[scale = .4]
    \draw[Orange] (0, 0) -- (6, 0) -- (6, 6) -- (0, 6) -- cycle;
    \draw[line width = 1pt] (1,0) -- (1, 6);
    \draw[line width = 1pt] (2, 0) -- (2, 6); 
    \draw[Purple, line width = 1pt] (0, 3) -- (6, 3);
    \draw[line width =1pt] (4, 0) -- (4, 6);
    \draw[line width = 1pt] (5, 0) -- (5, 6);

    \draw[Purple] (3,3) -- (3, 3.3);
    \draw (1, 1.5) -- (1.3, 1.5);
    \draw (1, 4.5) -- (1.3, 4.5);
    \draw (2, 1.5) -- (2.3, 1.5);
    \draw (2, 4.5) -- (2.3, 4.5);
    \draw (4, 1.5) -- (3.7, 1.5);
    \draw (4, 4.5) -- (3.7, 4.5);

    \draw (5, 1.5) -- (4.7, 1.5);
    \draw (5, 4.5) -- (4.7, 4.5);

    \draw[ -Stealth] (7, 3) -- (11, 3); 
    \node[anchor = south] at (9, 3) {\small \text{mutate along}};
    \node[anchor = north] at (9, 3) {\small \text{ \textcolor{Purple}{purple}}};

    \begin{scope} [shift = {(12, 0)}]
    \draw[Orange] (0, 0) -- (6, 0) -- (6, 6) -- (0, 6) -- cycle;
    \draw[line width = 1pt] (1,0) -- (1, 6);
    \draw[line width = 1pt] (2, 0) -- (2, 6); 
    \draw[Purple, line width = 1pt] (0, 3) -- (6, 3);
    
    \draw[line width = 1pt] (0, 1) -- (5, 6);
    \draw[line width = 1pt] (1, 0) -- (6, 5);
    \draw[line width = 1pt] (0, 5) -- (1, 6);
    \draw[line width = 1pt] (5, 0) -- (6, 1);

    \draw[Purple] (3,3) -- (3, 2.7);
    
    \draw (1, 1.5) -- (1.3, 1.5);
    \draw (1, 4.5) -- (1.3, 4.5);
    \draw (2, 1.5) -- (2.3, 1.5);
    \draw (2, 4.5) -- (2.3, 4.5);
    
    \draw (3.5, 4.5) -- (3.2, 4.8);
    \draw (5, 4) -- (4.7, 4.3);

    \draw[ -Stealth] (7, 3) -- (11, 3); 
    \node[anchor = south] at (9, 3) {\small \(\mathrm{GL}(2, \bZ)\)};
   
    \end{scope}

    \begin{scope} [shift = {(24, 0)}]
    \draw[Orange] (0, 0) -- (6, 0) -- (6, 6) -- (0, 6) -- cycle;
    \draw[line width = 1pt] (1,0) -- (1, 6);
    \draw[line width = 1pt] (2, 0) -- (2, 6); 
    \draw[Purple, line width = 1pt] (6, 0) -- (0, 6);
    
    \draw[line width = 1pt] (0, 1) -- (6, 1);
    \draw[line width = 1pt] (0, 2) -- (6, 2);

    \draw[Purple] (3, 3) -- (2.8, 2.8);

    \draw (1, 4.5) -- (1.3, 4.5);
    \draw (2, 4.5) -- (2.3, 4.5);
    
    \draw (4.5, 1) -- (4.5, 1.3);
    \draw (4.5, 2) -- (4.5, 2.3);

    \end{scope}
    \end{tikzpicture}.
    \]

    For \(\bM_\PIV^\gen\), consider:
    \[
    \begin{tikzpicture}[scale = .4]
    \draw[Orange] (0, 0) -- (6, 0) -- (6, 6) -- (0, 6) -- cycle;
    \draw[line width = 1pt] (1,0) -- (1, 6);
    \draw[line width = 1pt] (2, 0) -- (2, 6); 
    \draw[Purple, line width = 1pt] (0, 3) -- (6, 3);

    \draw[line width = 1pt] (5, 0) -- (5, 6);

    \draw[Purple] (3,3) -- (3, 3.3);
    \draw (1, 1.5) -- (1.3, 1.5);
    \draw (1, 4.5) -- (1.3, 4.5);
    \draw (2, 1.5) -- (2.3, 1.5);
    \draw (2, 4.5) -- (2.3, 4.5);

    \draw (5, 1.5) -- (4.7, 1.5);
    \draw (5, 4.5) -- (4.7, 4.5);

    \draw[ -Stealth] (7, 3) -- (11, 3); 
    \node[anchor = south] at (9, 3) {\small \text{mutate along}};
    \node[anchor = north] at (9, 3) {\small \text{ \textcolor{Purple}{purple}}};

    \begin{scope} [shift = {(12, 0)}]
    \draw[Orange] (0, 0) -- (6, 0) -- (6, 6) -- (0, 6) -- cycle;
    \draw[line width = 1pt] (1,0) -- (1, 6);
    \draw[line width = 1pt] (2, 0) -- (2, 6); 
    \draw[Purple, line width = 1pt] (0, 3) -- (6, 3);

    \draw[line width = 1pt] (1, 0) -- (6, 5);
    \draw[line width = 1pt] (0, 5) -- (1, 6);

    \draw[Purple] (3,3) -- (3, 2.7);
    
    \draw (1, 1.5) -- (1.3, 1.5);
    \draw (1, 4.5) -- (1.3, 4.5);
    \draw (2, 1.5) -- (2.3, 1.5);
    \draw (2, 4.5) -- (2.3, 4.5);

    \draw (5, 4) -- (4.7, 4.3);

     \draw[ -Stealth] (7, 3) -- (11, 3); 
    \node[anchor = south] at (9, 3) {\small \(\mathrm{GL}(2, \bZ)\)};
   
    \end{scope}

    \begin{scope} [shift = {(24, 0)}]
    \draw[Orange] (0, 0) -- (6, 0) -- (6, 6) -- (0, 6) -- cycle;
    \draw[line width = 1pt] (1,0) -- (1, 6);
    \draw[line width = 1pt] (2, 0) -- (2, 6); 
    \draw[Purple, line width = 1pt] (6, 0) -- (0, 6);
    
    \draw[line width = 1pt] (0, 1) -- (6, 1);

    \draw[Purple] (3, 3) -- (2.8, 2.8);

    \draw (1, 4.5) -- (1.3, 4.5);
    \draw (2, 4.5) -- (2.3, 4.5);
    
    \draw (4.5, 1) -- (4.5, 1.3);

    \end{scope}
    \end{tikzpicture}.
    \]

    For \(\PVdeg/\PIIIDsix\), consider
     \[
    \begin{tikzpicture}[scale = .4]
    \draw[Orange] (0, 0) -- (6, 0) -- (6, 6) -- (0, 6) -- cycle;
    
    \draw[line width = 1pt] (1,0) -- (1, 6);
    \draw (1, 4) -- (1.3, 4);

    \draw[line width = 1pt, blue] (2, 0) -- (2, 6);
    \draw[blue] (2, 4) -- (2.3, 4);
    
    \draw[Purple, line width = 1pt] (0, 2) -- (6, 2);
    \draw[Purple] (4, 2) -- (4, 2.3);

    \draw[line width = 1pt] (0, 1) -- (6, 1);
    \draw (4, 1) -- (4, 1.3);

    \draw[ -Stealth] (7, 3) -- (11, 3); 
    \node[anchor = south] at (9, 3) {\small \text{mutate along}};
    \node[anchor = north] at (9, 3) {\small \text{ \textcolor{Purple}{purple}}};

    \begin{scope} [shift = {(12, 0)}]
    \draw[Orange] (0, 0) -- (6, 0) -- (6, 6) -- (0, 6) -- cycle;
    
    \draw[line width = 1pt] (1,0) -- (1, 6);
    \draw (1, 4) -- (1.3, 4);

    \draw[line width = 1pt, blue] (2, 0) -- (2, 6);
    \draw[blue] (2, 4) -- (2.3, 4);
      
    \draw[Purple, line width = 1pt] (0, 2) -- (6, 2);
    \draw[Purple] (4, 2) -- (4, 1.7);

    \draw[line width = 1pt] (0, 1) -- (6, 1);
    \draw (4, 1) -- (4, 1.3);

    \draw[ -Stealth] (7, 3) -- (11, 3); 
    \node[anchor = south] at (9, 3) {\small \text{mutate along}};
    \node[anchor = north] at (9, 3) {\small \text{ \textcolor{blue}{blue}}};

    \end{scope}

     \begin{scope} [shift = {(24, 0)}]
    \draw[Orange] (0, 0) -- (6, 0) -- (6, 6) -- (0, 6) -- cycle;
    
    \draw[line width = 1pt] (1,0) -- (1, 6);
    \draw (1, 3) -- (1.3, 3);

    \draw[line width = 1pt, blue] (2, 0) -- (2, 6);
    \draw[blue] (2, 3) -- (1.7, 3);
      
    \draw[Purple, line width = 1pt] (0, 2) -- (6, 2);
    \draw[Purple] (3, 2) -- (3, 1.7);

    \draw[line width = 1pt] (0, 6) -- (6, 0);
    \draw (3, 3) -- (3.2, 3.2);

    \draw[ -Stealth] (7, 3) -- (11, 3); 
    \node[anchor = south] at (9, 3) {\small \(\mathrm{GL}(2, \bZ)\)};
    \end{scope}

    \begin{scope} [shift = {(36, 0)}]
    \draw[Orange] (0, 0) -- (6, 0) -- (6, 6) -- (0, 6) -- cycle;
    
    \draw[line width = 1pt] (4,0) -- (4, 6);
    \draw (4, 1) -- (4.3, 1);

    \draw[line width = 1pt, blue] (1, 0) -- (6, 5);
    \draw[line width = 1pt, blue] (0, 5) -- (1, 6);
    \draw[blue] (3.5, 2.5) -- (3.3, 2.7);
      
    \draw[Purple, line width = 1pt] (0, 2) -- (6, 2);
    \draw[Purple] (5, 2) -- (5, 1.7);

    \draw[line width = 1pt] (0, 1) -- (5, 6);
    \draw[line width = 1pt] (5, 0) -- (6, 1);     
    \draw (2.5, 3.5) -- (2.7, 3.3);
    \end{scope}

    \end{tikzpicture}.
    \]

    For \(\PIIIDseven\), consider 
    \[
    \begin{tikzpicture}[scale = .4]
    \draw[Orange] (0, 0) -- (6, 0) -- (6, 6) -- (0, 6) -- cycle;
    \draw[line width = 1pt] (1,0) -- (1, 6);
    
    \draw[Purple, line width = 1pt] (6, 0) -- (0, 6);
    
    \draw[line width = 1pt] (0, 1) -- (6, 1);

    \draw[Purple] (3, 3) -- (3.2, 3.2);

    \draw (1, 3) -- (1.3, 3);
    \draw (3, 1) -- (3, 1.3);

    \draw[ -Stealth] (8, 3) -- (12, 3); 
    \node[anchor = south] at (10, 3) {\small \text{mutate along}};
    \node[anchor = north] at (10, 3) {\small \text{ \textcolor{Purple}{purple}}};

    \begin{scope} [shift = {(14, 0)}]
    \draw[Orange] (0, 0) -- (6, 0) -- (6, 6) -- (0, 6) -- cycle;
    \draw[line width = 1pt] (1,0) -- (1, 6);
    \draw (1, 3) -- (1.3, 3);
    
    \draw[Purple, line width = 1pt] (6, 0) -- (0, 6);
    
    \draw[line width = 1pt] (0, 3) -- (6, 0);
    \draw[line width = 1pt] (0, 6) -- (6, 3); 

    \draw[Purple] (3, 3) -- (2.8, 2.8);

    \draw (3, 1.5) -- (3.15, 1.8);
    \draw (3, 4.5) -- (3.15, 4.8);

    \end{scope}
    \end{tikzpicture}.
    \]    
\end{proof}